\documentclass[a4paper]{amsart}
\usepackage[utf8]{inputenc}
\usepackage[english]{babel}

\usepackage{amscd,amssymb}
\usepackage{amsmath}
\usepackage{amsthm}
\usepackage{graphicx}
\usepackage{fancybox}
\usepackage{epic,eepic}
\usepackage{amstext}
\usepackage{mathtools}
\usepackage{esint}
\usepackage[square,numbers]{natbib}
\usepackage{booktabs}
\usepackage[hide links]{hyperref}
\usepackage{comment}
\usepackage{mathtools}
\usepackage{pgfplots}
\usepackage{todonotes}
\usepackage{tabularx}

\usepackage{subfloat}
\usepackage{subcaption}
\usepackage{bm}

\usepackage[noabbrev, capitalise, nameinlink]{cleveref}
\crefname{equation}{}{}
\crefname{assumption}{Assumption}{Assumptions}
\crefformat{equation}{\textup{#2(#1)#3}}

\newtheorem{theorem}{Theorem}[section]
\newtheorem{corollary}[theorem]{Corollary}
\newtheorem{lemma}[theorem]{Lemma}
\newtheorem{proposition}[theorem]{Proposition}

\theoremstyle{definition}

\theoremstyle{remark}
\newtheorem{remark}[theorem]{Remark}
\numberwithin{theorem}{section}
\numberwithin{equation}{section}
\numberwithin{figure}{section}

\def\supp{\operatorname{supp}}

\def\with{\,:\,}
\def\dx{\,\text{d}x}

\newcommand{\cAdd}{\color{black}}

\numberwithin{equation}{section}
\numberwithin{theorem}{section}

\AtBeginDocument{%
	\def\MR#1{}
}

\begin{document}
	
\title[Optimal Spectral Approximation in the Overlaps for GFEMs]{Optimal Spectral Approximation in the Overlaps for Generalized Finite Element Methods}
\author[C.~Alber, P.~Bastian, M.~Hauck, R.~Scheichl]{Christian Alber$^*$, Peter Bastian$^\dagger$, Moritz Hauck$^\ddagger$,  Robert Scheichl$^{*,\dagger}$}
\address{${}^{*}$ Institute for Mathematics, Heidelberg University, 69047 Heidelberg, Germany}
\email{\{c.alber, r.scheichl\}@uni-heidelberg.de}
\address{${}^\dagger$ Interdisciplinary Center for Scientific Computing (IWR), Heidelberg University, 69047 Heidelberg, Germany}
\email{peter.bastian@iwr.uni-heidelberg.de}
\address{${}^\ddagger$ Institute for Applied and Numerical Mathematics, Karlsruhe Institute of Technology, Englerstr.~2, 76131 Karlsruhe, Germany}
\email{moritz.hauck@kit.edu}

	
\begin{abstract}
In this paper, we study a generalized finite element method for solving second-order elliptic partial differential equations with rough coefficients. The method uses local approximation spaces computed by solving eigenvalue problems on rings around the boundary of local subdomains. Compared to the corresponding method that solves eigenvalue problems on the whole subdomains, the problem size and the bandwidth of the resulting system matrices are substantially reduced, resulting in faster spectral computations. We prove a nearly exponential a priori decay result for the local approximation errors of the proposed method, which implies the nearly exponential decay of the overall approximation error of the method. The proposed method can also be used as a preconditioner, and only a slight adaptation of our theory is necessary to prove the optimal convergence of the preconditioned iteration.  Numerical experiments are presented to support the effectiveness of the proposed method and to investigate its coefficient robustness. 
\end{abstract}



\maketitle

\renewcommand{\u}{\boldsymbol{u}}
\renewcommand{\v}{\boldsymbol{v}}
\newcommand{\f}{\boldsymbol{f}}
\newcommand{\p}{p}
\newcommand{\q}{q}

\newcommand{\uh}{\boldsymbol{u}_h}
\newcommand{\vh}{\boldsymbol{v}_h}
\newcommand{\vhi}{\boldsymbol{v}_{h,i}}
\newcommand{\ph}{p_h}
\newcommand{\qh}{q_h}

\newcommand{\uph}[1]{\u_{h,{#1}}^p}
\newcommand{\uhp}{\u_h^p}
\newcommand{\uphj}{\uph{j}}

\newcommand{\numax}{\nu_{\mathrm{max}}}
\newcommand{\numin}{\nu_{\mathrm{min}}}

\newcommand{\ue}{\u^e}
\newcommand{\pe}{\p^e}
\newcommand{\uhe}{\u_h^e}
\newcommand{\phe}{\p_h^e}

\renewcommand{\L}[1]{L^2({#1})}
\newcommand{\Winf}[1]{W^{1,\infty}(#1)}
\newcommand{\Lz}[1]{L_0^2(#1)}
\renewcommand{\H}[1]{H^1(#1)}
\newcommand{\Hz}[1]{H_0^1(#1)}
\newcommand{\Linf}[1]{L^{\infty}(#1)}
\newcommand{\aharmg}[1]{H_{a_{\gamma}}(#1)}
\newcommand{\aharm}[1]{H_{a}(#1)}
\newcommand{\aharmZ}[1]{H_{a,0}(#1)}
\newcommand{\hdi}[3]{\boldsymbol{H}_{#1}(\di^{#2},{#3})}
\newcommand{\hcu}[3]{\boldsymbol{H}_{#1}(\cur^{#2},{#3})}
\newcommand{\hcuo}[2]{\boldsymbol{H}_{#1}^1(\operatorname{curl},{#2})}
\newcommand{\hdivO}{\boldsymbol{H}(\di,\Omega)}
\newcommand{\hdivOZ}{\boldsymbol{H}(\di^0,\Omega)}
\newcommand{\hdivON}{\boldsymbol{H}_0(\di,\Omega)}
\newcommand{\hdivoi}{\boldsymbol{H}(\di,\omega_i^*)}
\newcommand{\hdivoiN}{\boldsymbol{H}^{\operatorname{div}}_{N}(\omega_i^*)}
\newcommand{\hdivoiNZ}{\boldsymbol{H}^{\operatorname{div},0}_{N}(\omega_i^*)}
\newcommand{\hdivoiNI}{\boldsymbol{H}^{\operatorname{div}}_{NI}(\omega_i^*)}
\newcommand{\hdivoiNIZ}{\boldsymbol{H}^{\operatorname{div},0}_{NI}(\omega_i^*)}

\newcommand{\linop}[2]{\mathcal{L}(#1,#2)}
\newcommand{\htrace}[1]{H^{\frac{1}{2}}(\partial #1)}

\newcommand{\mesh}{\mathbb{T}_h}
\newcommand{\bdm}[1]{\mathcal{BDM}_0(#1, \mesh)}
\newcommand{\pdisc}[2]{\mathcal{P}^{\mathrm{disc}}_{#2}(#1, \mesh)}
\newcommand{\pcont}[2]{\mathcal{P}_{#2}(#1, \mesh)}
\newcommand{\Vh}[1]{\boldsymbol{V}_h(#1)}
\newcommand{\Vhz}[1]{\boldsymbol{V}_{h,0}(#1)}
\newcommand{\aharmgh}[1]{H_{a_{\gamma},h}(#1)}
\newcommand{\interop}{I_h}
\newcommand{\partoph}[1]{P_{h,#1}}

\newcommand{\trace}{T}
\newcommand{\ext}{E}

\renewcommand{\a}[3]{a_{#1}(#2,#3)}
\newcommand{\ag}[3]{a_{\gamma,#1}(#2,#3)}
\newcommand{\agh}[3]{a_{\gamma,#1, h}(#2,#3)}
\renewcommand{\b}[3]{b_{#1}(#2,#3)}
\newcommand{\scal}[3]{(#1,#2)_{\L{#3}}}
\newcommand{\ao}{a_{\omega_i^*}}
\newcommand{\bo}{b_{\omega_i^*}}
\newcommand{\intOm}{\int_{\Omega}}

\newcommand{\ahi}{a_h^i}
\newcommand{\api}{a_p^i}
\newcommand{\aci}{a_c^i}
\newcommand{\ahb}{a_h^{\partial}}
\newcommand{\apb}{a_p^{\partial}}
\newcommand{\acb}{a_c^{\partial}}
\newcommand{\ah}{a_{h}}
\newcommand{\sumdg}[1]{\llbracket {#1} \rrbracket}
\newcommand{\sumdgw}[1]{\llbracket {#1} \rrbracket_w}
\newcommand{\jumpdg}[1]{\sumdg{{#1} \otimes \normal}}
\newcommand{\scaldg}[3]{{\langle {#1}, {#2} \rangle_{#3}}}
\newcommand{\ahosj}{a_{h,\omega_j^*}}
\newcommand{\ahosji}{a_{h,\omega_j^*}^i}
\newcommand{\ahosjb}{a_{h,\omega_j^*}^\partial}
\newcommand{\ahos}{a_{h,\omega^*}}
\newcommand{\ahosi}{a_{h,\omega^*}^i}
\newcommand{\aposi}{a_{p,\omega^*}^i}
\newcommand{\acosi}{a_{c,\omega^*}^i}
\newcommand{\ahosb}{a_{h,\omega^*}^\partial}

\newcommand{\facedg}{\mathbb{F}_h}
\newcommand{\facedgint}{\mathbb{F}_h^i}
\newcommand{\facedgbdry}{\mathbb{F}_h^\partial}
\newcommand{\dgstab}{\gamma_h^2}
\newcommand{\tmin}{T_{\mathrm{min}}}
\newcommand{\tmax}{T_{\mathrm{max}}}
\newcommand{\reconop}{Q_h}
\newcommand{\reconopvec}{Q_h^d}
\newcommand{\face}{F}

\newcommand{\lnorm}[2]{\|#1\|_{\L{#2}}}
\newcommand{\agnorm}[2]{\|#1\|_{a_{\gamma}, #2}}
\newcommand{\anorm}[2]{\|#1\|_{a, #2}}
\newcommand{\hdinorm}[2]{\|{#1}\|_{\hdi{}{}{{#2}}}}
\newcommand{\norm}[2]{\|{#1}\|_{#2}}
\newcommand{\seminorm}[2]{|{#1}|_{#2}}
\newcommand{\hnorm}[2]{\norm{#1}{H^1({#2})}}
\newcommand{\hsemnorm}[2]{|#1|_{H^1({#2})}}
\newcommand{\sipnorm}[2]{\norm{#1}{SIP,{#2}}}
\newcommand{\sipnormi}[2]{\sipnorm{#1}{#2,i}}
\newcommand{\jumpnorm}[2]{|#1|_{J, #2}}

\newcommand{\di}{\operatorname{div}}
\newcommand{\cur}{\operatorname{curl}}
\newcommand{\cu}{\nabla\times}
\newcommand{\cuScal}{\boldsymbol{\operatorname{curl}}}
\newcommand{\normal}{\boldsymbol{n}}

\newcommand{\inte}[3]{\int_{#1}{#2} \medspace d\boldsymbol{#3}}

\newcommand{\M}{M}
\newcommand{\Mg}{\M_{\gamma}}
\newcommand{\A}{A}
\newcommand{\Ag}{\A_{\gamma}}
\newcommand{\B}{B}
\renewcommand{\P}{P}
\newcommand{\Schur}{S}
\newcommand{\Schurg}{\Schur_{\gamma}}
\newcommand{\Aap}{\hat{A}}
\newcommand{\Schurap}{\hat{S}}
\newcommand{\Schurapg}{\Schurap_{\gamma}}
\newcommand{\Pap}{\hat{P}}
\newcommand{\W}{W}

\newcommand{\meta}[1]{\textcolor{blue}{#1}}
\newcommand{\ca}[1]{{\color{red} #1}}

\newcommand{\lorth}{P}
\newcommand{\orthProj}[1]{\mathbb{P}_{#1}}

\newcommand{\om}{\omega}
\newcommand{\os}{\omega^{\ast}}
\newcommand{\ommin}{\tilde \om}
\newcommand{\osmin}{\tilde \omega^*}
\newcommand{\ring}{R}
\newcommand{\ringos}{R^{\ast}}
\newcommand{\ringeta}{R^{\eta}}

\newcommand{\omi}{\omega_i}
\newcommand{\osi}{\omega^{\ast}_i}
\newcommand{\ommini}{\tilde \om_i}
\newcommand{\osmini}{\tilde \omega^*_i}
\newcommand{\ringi}{R_i}
\newcommand{\ringosi}{R^{\ast}_i}
\newcommand{\ringetai}{R^{\eta}_i}

\newcommand{\nwidthh}[1]{d_{h,#1}(\os{#1},\om{#1})}
\newcommand{\numDom}{M}
\newcommand{\nloc}[1]{n_{loc}}
\newcommand{\gfemSpace}[2]{S^{#2}_{\nloc{#1}}(\om{#1})}
\newcommand{\gfemSpaceInt}[2]{\hat{S}^{#2}_{\nloc{#1}}(\om{#1})}
\newcommand{\omcon}{\kappa}
\newcommand{\oscon}{\kappa^{\ast}}
\newcommand{\pu}{P}
\newcommand{\AharmEx}{L}

\newcommand{\specProj}{P_J}
\newcommand{\z}{\boldsymbol{z}}
\newcommand{\y}{\boldsymbol{y}}
\newcommand{\C}{\mathbb{C}}
\newcommand{\R}{\mathbb{R}}
\newcommand{\N}{\mathbb{N}}


\section{Introduction}

This paper considers the numerical solution of second-order elliptic partial differential equations (PDEs) with strongly heterogeneous and highly varying (non-periodic) coefficients. Such so-called multiscale problems arise, for example, when modeling fluid flow in porous media or simulating composite materials. The numerical treatment of multiscale problems with classical finite element methods (FEMs) suffers from suboptimal approximation rates and preasymptotic effects on meshes that do not resolve the microscopic details of the coefficients, cf.~\cite{Babuka1999}. Consequently, very fine meshes are required, leading to large and possibly ill-conditioned  systems of equations. This motivates the development of multiscale methods that achieve accurate approximations already at coarse scales. To this end, they incorporate physically important microscopic features into the coarse approximation~space.

The construction and analysis of multiscale methods has been an active field in the last decades. In the following, we distinguish between two classes of methods: First, methods that exploit structural properties of the coefficients, such as periodicity and scale separation, to construct the problem-adapted basis functions. Their computational cost typically differs from that of classical FEMs on the same mesh only by the cost of solving a fixed number of local problems.  This class includes the Heterogeneous Multiscale Method~\cite{EE03}, the Two-Scale Finite Element Method \cite{MaS02}, and the Multiscale Finite Element Method~\cite{HoW97}.
In contrast, methods of the second class provide accurate approximations  under minimal structural assumptions on the coefficients. This is achieved at the expense of a moderate computational overhead compared to classical FEMs. This overhead is manifested in an extended support of basis functions or an increased number of basis functions per mesh entity.  Prominent methods include the Generalized Multiscale Finite Element Method~\cite{EfeGH13,ChuEL18b}, the Multiscale Spectral Generalized Finite Element Method (MS-GFEM) \cite{BabL11, Ma22}, Adaptive Local Bases~\cite{GraGS12}, the \text{(Super-)} Localized Orthogonal Decomposition (LOD) method~\cite{MalP14,HenP13,HaPe21b,Freese2023}, or Gamblets~\cite{Owh17}; for a more comprehensive overview of multiscale methods see the recent review articles~\cite{AltHP21,Ma24}. 

The method presented in this paper is a variant of the MS-GFEM, which was originally introduced in \cite{BabL11} and then further developed in \cite{babuvska2014machine,babuvska2020multiscale,Ma22,ma2022error,Ma24}. Note that there are also variants of the method that compute close to optimal approximation spaces at lower cost, based on the work~\cite{Buhr2018}. These methods belong to the class of Generalized Finite Element Methods (GFEMs), which is an extension of the classical FEM methodology, cf.~\cite{babuvska1997partition,melenk1995generalized}. Given a partition of unity and corresponding overlapping subdomains, the GFEM employs possibly non-polynomial local approximation spaces on each subdomain. These local spaces are then ``glued'' together using the partition of unity, resulting in a global approximation space. The MS-GFEM employs optimal local approximation spaces, which are computed by solving local spectral problems in the space of locally operator-harmonic functions. Note that the local spectral problems are mutually independent and therefore can be solved in parallel. However, they still pose a significant computational challenge due to their complex saddle-point structure, where a Lagrange multiplier is used to impose the harmonicity constraint.

In this work, we aim to reduce the size of the local spectral problems and thereby reduce the computational cost of solving them. This is achieved by localizing the spectral computations to a ring around the boundary of the subdomains. The local optimal approximation space on the ring is then operator-harmonically extended to the whole subdomain. 
{\color{black}
We emphasize that the computational savings are not only a result of the reduced number of degrees of freedom on the ring, but also due to the better sparsity pattern of the matrices associated with the problem on the ring. Specifically, the sparsity pattern resembles that of a spatially lower-dimensional problem, allowing for a fast factorization with less fill-in. This effect is clearly visible in our numerical experiments.}
Note that related localization ideas have been used in the context of spectral coarse spaces in domain decomposition methods; see, e.g., \cite{Dohrmann2015,Heinlein2019}, where the resulting methods were referred to as economic versions. We provide a rigorous theoretical justification of the localization and derive nearly exponential approximation results at the local and global levels. 
Such an a priori analysis is new and, for example, the analysis in \cite{Dohrmann2015,Heinlein2019} was restricted to a posteriori results. 
We emphasize that the proven decay results are qualitatively similar to those of the MS-GFEM, where spectral problems are solved on the entire subdomains. However, due to the limited information available when considering only the ring, it is expected that more eigenfunctions are required for each subdomain.
{\color{black}
	 This effect, which can also be observed numerically, is theoretically reflected in larger constants in the decay estimates. Nevertheless, the computational advantage of using rings, due to a reduced number of degrees of freedom and an improved sparsity pattern, remains.}
For the sake of brevity, only the infinite dimensional setting will be discussed,
even though the speedup is mainly of interest in the discrete setting. 
Our methodology and analysis generalize directly to the discrete setting and we 
will point out the minor changes required.

Building on the observation in \cite{Strehlow2024} that the MS-GFEM can be used as a two-level Restricted Additive Schwarz (RAS) \cite{Cai1999} preconditioner within an iterative method, we propose a similar approach for our method. 
Note that, more generally, spectral coarse spaces are a popular tool for achieving coefficient-robust convergence properties, and many different constructions have been developed. We refer for example to the Finite Element Tearing and Interconnecting (FETI) and Neumann-Neumann methods \cite{Bjrstad2001,Spillane2013}, the Balancing Domain Decomposition by Constraints and FETI Dual-Primal methods \cite{Mandel2007, Mandel2012,Klawonn2016,Pechstein2017,Kim2017}, as well as the Overlapping Schwarz methods \cite{galvis2010domain-reduced,Nataf2011,efendiev2012robust,spillane2014abstract,gander2015analysis,Heinlein2019,heinlein2022fully,bastian2022multilevel,al2023efficient}.
The spectral coarse space of the above mentioned RAS-type method is precisely the approximation space of the multiscale method.
Eventually, when the problem size becomes larger or an improved accuracy is required, the coarse problem of the proposed method becomes too large to be solved with a direct solver. In this case, an iterative version of the proposed method is particularly attractive because by using multiple iterations, one can achieve high accuracy with a significantly smaller coarse space. The theoretical framework used to analyze the proposed method as a multiscale method directly implies the convergence of the preconditioned iteration. Our numerical experiments show a quantitatively similar convergence behavior for the preconditioned iteration based on our proposed method on rings and the original MS-GFEM. However, 
{ since the eigenproblem in the ring-based method does not take into account information about the coefficient outside the ring}, it is possible to design high-contrast channels that require a larger number of modes to be selected compared to the MS-GFEM.

This paper is organized as follows. First, we introduce the model problem in \cref{sec:modelproblem}, followed by an abstract presentation of the GFEM in \cref{sec:gfem}. In \cref{sec:construction} we present the construction of the proposed method, specifying its local particular functions and local approximation spaces. An a priori convergence analysis of the proposed method is given in \cref{sec:convergence}. In \cref{sec:preconditioner} we use the proposed method as a preconditioner and show the convergence of the preconditioned iteration. Finally, we conclude the paper with a series of numerical experiments in \cref{sec:numericalexperiments}.

\section{Model problem}
\label{sec:modelproblem}
In this section, we introduce a prototypical second-order linear elliptic PDE, which will serve as the model problem for this paper. Note that, as demonstrated in \cite{Ma24}, the approach of the presented method and its analysis can be easily extended to PDEs beyond this model problem. Given a polygonal Lipschitz domain $\Omega \subset \mathbb R^d$, $d \in \{2,3\}$, the model problem seeks a solution $u \colon \Omega \to \mathbb R$ satisfying
\begin{equation}
	\label{eq:PDEstrong}
	\left\{
	\begin{aligned}
		- \operatorname{div} (\bm A \nabla u) &= f && \quad \text{in } \Omega,\\
		u &= g && \quad \text{on } \partial \Omega.
	\end{aligned}\right.
\end{equation}
The matrix-valued coefficient $\bm A\in L^\infty(\Omega,\mathbb{R}^{d\times d})$ is assumed to be symmetric and positive definite almost everywhere, i.e., there exist constants ${0<\alpha_\mathrm{min}\leq \alpha_\mathrm{max}< \infty}$ such that for almost all $\bm x \in \Omega$ and for all $\bm \xi\in \mathbb{R}^d$ it holds that
\begin{equation}\label{eq:propA}
	\alpha_\mathrm{min}|\bm \xi|^2 \leq (\bm A(\bm x)\bm \xi)\cdot \bm \xi \leq \alpha_\mathrm{max} |\bm \xi|^2,
\end{equation}
where $|\cdot|$ denotes the Euclidean norm of a $d$-dimensional vector. For the source term and Dirichlet data, we assume that $f \in L^2(\Omega)$ and $g \in H^{1/2}(\partial \Omega)$, respectively.

The weak formulation of~\cref{eq:PDEstrong} is based on the Sobolev space $H^1(\Omega)$ and its affine subspace of functions with  Dirichlet datum $g \in H^{1/2}(\partial \Omega)$  defined as 
\begin{equation}
H^1_{g}(\Omega) \coloneqq \{v \in H^1(\Omega) \with v|_{\partial \Omega} = g\},
\end{equation} 
where the latter restriction to the boundary has to be interpreted in the sense of traces.
Note that when setting $g$ to zero, we retrieve the subspace $H^1_0(\Omega)$ satisfying homogeneous Dirichlet boundary conditions. 
The bilinear and linear forms associated with the problem are for any $u \in H^1(\Omega)$ and $v \in H^1_0(\Omega)$ defined as
\begin{equation}
	\label{eq:defaandF}
	a(u, v) \coloneqq
	\int_\Omega (\bm A\nabla u)\cdot\nabla v\dx,\qquad 	F(v) \coloneqq \int_\Omega fv\dx.
\end{equation}

Given a right-hand side $f$ and Dirichlet data $g$, the weak formulation of problem~\cref{eq:PDEstrong} seeks a function $u \in H^1_g(\Omega)$ such that 
\begin{equation}
	\label{eq:PDEweak}
	a(u, v) = F(v) \qquad \forall v \in H^1_0(\Omega).  
\end{equation} 
By the Poincaré--Friedrichs inequality and the uniform coefficient bound \cref{eq:propA}, we obtain that $a$ is an inner product on the space~$H^1_0(\Omega)$ with induced norm $\|\cdot\|^2_a \coloneqq a(\cdot,\cdot)$ equivalent to the full $H^1(\Omega)$-norm. Therefore, using the fact that $F$ is a bounded linear functional, the Riesz representation theorem can be applied to prove the well-posedness of \cref{eq:PDEweak}.

For later use, we introduce restricted versions of~$a$ and~$F$ to a given subdomain $S \subset \Omega$, denoted by $a_S$ and $F_S$, respectively. Such restricted versions can be obtained by restricting the integrals in definition~\cref{eq:defaandF} to the subdomain~$S$. The corresponding restricted energy norm is denoted by $\|\cdot\|_{a,S}^2\coloneqq a_S(\cdot,\cdot)$.  

\section{Generalized finite element method}\label{sec:gfem}

The method proposed in this paper belongs to the class of Generalized Finite Element Methods (GFEMs). The GFEM is an extension of the classical FEM that allows the use of ansatz spaces adapted to the problem at hand, potentially leading to significantly more accurate approximations than those obtained using polynomial-based ansatz spaces. In the following, we will outline the general construction of the GFEM and provide a global approximation result for its discrete solution.
Let us consider an overlapping domain decomposition \(\{\omega_i\}_{i=1}^M\) of the computational domain \(\Omega\), composed of \(M \in \mathbb{N}\) open subdomains (\(\mathbb{N}\) denoting the set of positive natural numbers). We assume that \(\cup_{i=1}^{M} \omega_i = \Omega\), and that each point \(\bm x \in \Omega\) belongs to at most \(\kappa \in \mathbb N\) subdomains. Furthermore, we consider a partition of unity \(\{\chi_i\}_{i=1}^M\) subordinate to this decomposition, and assume that the partition of unity functions satisfy the following classical properties:
\begin{equation}
	\label{eq:propPU}
	\begin{aligned}
			0\leq \chi_{i}\leq 1,\qquad \operatorname{supp}(\chi_i) \subset \overline{\omega_i},\qquad  \sum_{i=1}^{M}\chi_{i} \equiv 1,\\
		\chi_{i}\in W^{1,\infty}(\Omega),\qquad  \max_{\bm x \in \Omega}\,|\nabla \chi_i|\leq \frac{C_\chi}{{\cAdd \delta_i}},
	\end{aligned}
\end{equation}
{\cAdd where $C_\chi>0 $ is a constant of order one, and $\delta_i$ denotes the minimal width of the overlap region $\{\bm x \in \omi : \exists j\neq i \text{ s.t. } \bm x \in \om_j\}$ of $\omi$.}

The construction of the GFEM begins by identifying suitable local approximation spaces \(S_{n_i}(\omega_i) \subset H^1(\omega_i)\) of dimension \(n_i \in \mathbb{N}\), which are assumed to satisfy homogeneous boundary conditions on \(\partial \omega_i \cap \partial \Omega\). To construct a global approximation space, the GFEM then ``glues'' these local spaces together using the partition of unity. More precisely, denoting by \(\bm{n} = (n_1, \dots, n_M) \in \mathbb{N}^M\) the vector containing the \(n_i\)'s, we define the global approximation space of the GFEM as
\begin{equation}
	\label{eq:Sn}
	S_{\bm n}(\Omega) \coloneqq \left\{ \sum_{i=1}^{M} \chi_i \phi_i : \, \phi_i \in S_{n_i}(\omega_i) \right\} \subset H_0^1(\Omega).
\end{equation}
As will be discussed later (cf.~\cref{thm:1-0}), the approximation error of this global space can be bounded by the maximum of the local approximation errors.

To handle inhomogeneous boundary data or source terms, the GFEM uses so-called particular functions. A global particular function $u^p$ can be constructed by gluing together local particular functions $u_i^p \in H^1(\omi)$ corresponding to the subdomains. For subdomains at the boundary, the $u_i^p$ are constructed such that $u_i^p = g$ on $\partial \omega_i \cap \partial \Omega$, while for interior subdomains homogeneous Dirichlet boundary conditions are imposed. 
The global particular function is then given by
\begin{equation*}
	u^p \coloneqq  \sum_{i=1}^{M} \chi_i u_i^p,
\end{equation*}
which satisfies $u^p \in H^1_g(\Omega)$ by the partition of unity property, cf.~\cref{eq:propPU}.

The GFEM solution is then defined as the Galerkin approximation in the affine space $u^p + S_{\bm n}(\Omega)$, i.e., we seek $u^G = u^p + u^s$ with $u^s \in S_{\bm n}(\Omega)$ such that 
\begin{equation}
	\label{eq:weakformulationgfem}
	a(u^s,v) = F(v) - a(u^p,v)\qquad \forall v \in S_{\bm n}(\Omega). 
\end{equation}

The following theorem provides a bound on the approximation error of the GFEM in terms of its local approximation errors. Such results are classical, and the specific version presented below, which includes the partition of unity functions in the local approximation estimates, can be found in~\cite[Thm.~2.1]{Ma22}.
\begin{theorem}[Error estimate of GFEM]\label{thm:1-0}
Given any function $\Psi \coloneqq u^{p} + \sum_{i=1}^{M}\chi_{i}\phi_{i}$, where $u^p \in H^1_g(\Omega)$ and the $ \phi_i \in S_{n_i}(\omi)$ satisfy
		\begin{equation}\label{eq:1-3-1}
		\big\Vert \chi_{i}(u-u^{p}_{i}-\phi_{i})\big\Vert_{a,\,\omega_{i}}\leq \varepsilon_{i} {\cAdd \big\Vert u \big\Vert_{a,\,\osi}},
	\end{equation}
	for some $\varepsilon_{i}>0$,  the GFEM error can be bounded as
	\begin{equation}
		\|u-u^G\|_a \leq \|u-\Psi\|_a
		\leq  \cAdd{\sqrt{\kappa\kappa^*}\left(\max_{i=1,\ldots,M} \varepsilon_i\right)\big\Vert u \big\Vert_{a}},
	\end{equation}
	{\cAdd where $\kappa^* \in \mathbb N$ is the maximum number of subdomains $\osi$ containing a point $\bm x \in \Omega$.}
	In particular, the GFEM approximation is optimal in the space $u^p + S_{\bm n}(\Omega)$.
\end{theorem}
 
\section{Construction of the proposed method}\label{sec:construction}

\begin{figure}\label{fig:1-1}
	\centering
	\includegraphics[width=.4\linewidth]{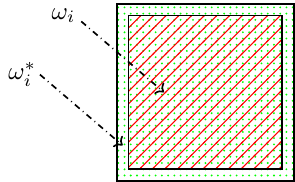}\hspace{.5cm}
	\includegraphics[width=.4\linewidth]{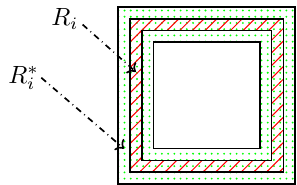}\hfill
	
	\caption{Illustration of the subdomains $\omi$ and $\osi$ (left) and  of the rings $\ringi$ and $\ringosi$ (right) in the case of an interior subdomain.}
	\label{fig:patchandring}
\end{figure}

This section is devoted to the precise construction of the proposed method, i.e., its choice of local particular functions and local approximation spaces. To simplify the presentation, we will first consider interior subdomains. The treatment of boundary subdomains will be discussed later in \cref{sec:boundarysubdomains}.

\subsection{Original MS-GFEM}
Since the proposed method is an enhancement of the methods in~\cite{BabL11,Ma22}, we first briefly recall the construction of the MS-GFEM. The MS-GFEM constructs local particular functions by solving a problem on an oversampling domain $\omega_i^*$ with \(\omega_i \subset \omega_i^* \subset \Omega\); see \cref{fig:patchandring} (left) for an illustration. Specifically, the problem seeks \(\psi_i \in H^1_0(\omega_i^*)\) such that
\begin{equation}
	\label{eq:locparticularfun}
	a_{\omega_i^*}(\psi_i, v) = F_{\omega_i^*}(v)\qquad \forall v \in H^1_0(\omega_i^*).
\end{equation}
The local particular function is chosen as $u_i^p \coloneqq \psi_i|_{\omega_i}$.

To construct the local approximation space of the MS-GFEM, one uses the observation that the difference \(u|_{\osi} - \psi_i\) is an operator-harmonic function, i.e., it is an element of the space \(H_a(\omega_i^*)\), defined as
\begin{equation}
	\label{eq:harmonicfunctions}
	H_a(\omega_i^*) \coloneqq \left\{ v \in H^1(\omega_i^*) \with  a_{\omega_i^*}(v, \varphi) = 0 \;\; \forall \varphi \in H^1_0(\omega_i^*) \right\}.
\end{equation}
The local approximation space is then defined as $S_{n_i}(\omi) \coloneqq \operatorname{span}\{u_1|_{\omi},\dots,u_{n_i}|_{\omi}\}$, where 
$(u_k,\lambda_k) \in H_{a}(\omega_i^*) \times \mathbb R$ are the first $n_i$ eigenpairs of   
\begin{equation}
	\label{eq:evpharmonicfullincludingconstant}
	a_{\omega_i^*}(u_k, \varphi) = \lambda_k\, a_{\omega_i}(\chi_i u_k, \chi_i \varphi)\qquad \forall\varphi \in H_{a}(\omega_i^*). 
\end{equation}
On the space $H_a(\omega_i^*)$,
the restricted energy norm $\|\cdot\|_{a,\osi}$ is a seminorm, and to obtain a subspace on which it is also a norm, we introduce the bounded linear functional  
\begin{equation}
	\label{eq:defMi}
	\mathcal M_i\colon H^1(\osi)\to \mathbb R,\quad v \mapsto  \bigg(\int_{\omega_i} A \nabla \chi_i \cdot \nabla \chi_i \dx\bigg)^{-1}\int_{\omega_i} A \nabla (\chi_i v) \cdot \nabla \chi_i \dx.
\end{equation}
In fact, defining
\begin{equation}
	\label{eq:harmonicfunctions0}
	H_{a,0}(\omega_i^*) \coloneqq \left\{ v \in H_a(\omega_i^*) \with  \mathcal M_iv = 0 \right\}
\end{equation}
 gives a Hilbert space when equipped with the inner product~$a_{\omi}(\cdot,\cdot)$. Note that it holds $u|_{\osi} - \psi_i-\mathcal M_i (u|_{\osi} - \psi_i) \in H_{a,0}(\osi)$ by the definition of $\mathcal M_i$.
The construction of the local approximation space of the MS-GFEM is then based on the observation that, after multiplication with the partition of unity function $\chi_i$, the latter function is contained in the range of the compact restriction operator
\begin{equation}
	\label{eq:defP}
	P_i \colon H_{a,0}(\osi) \to H^1_0(\omega_i), \qquad v \mapsto (\chi_i v)|_{\omega_i}.
\end{equation}
The compactness of this operator can be proved using a Caccioppoli-type inequality and the compactness of the embedding \( H^1(\omega^*) \hookrightarrow L^2(\omega^*) \); see, e.g., \cite[Lem.~3.1]{Ma22}. The Kolmogorov \(n\)-width  of the operator $P_i$, which is a measure for the approximability of the range using finite dimensional subspaces, is then defined as
\begin{equation}
	\label{eq:nwidthP}
	d_{n}(\omega_i, \omega_i^*) \coloneqq \inf_{Q_i(n) \subset H_0^1(\omega_i)} \sup_{u \in H_{a, 0}(\omega_i^*)} \inf_{v \in Q_i(n)} \frac{\|P u - v\|_{a, \omega_i}}{\|u\|_{a, \omega_i^*}},
\end{equation}
where the infimum is taken over subspaces \( Q_i(n) \subset H^1_0(\omega_i) \) of dimension $n$. The optimal subspace can be characterized with the help of an eigenproblem involving the operator $P_i^*P_i$ with $P_i^*$ denoting the adjoint of the operator $P_i$. It seeks eigenpairs  
$(v_k,\lambda_k) \in H_{a,0}(\omega_i^*) \times \mathbb R$ such that  
\begin{equation}
	\label{eq:evpharmonicfull}
	a_{\omega_i^*}(v_k, \varphi) = \lambda_k\, a_{\omega_i}(\chi_i v_k, \chi_i \varphi)\qquad \forall\varphi \in H_{a,0}(\omega_i^*). 
\end{equation}
The optimal subspace of dimension $n$ is then given by
\begin{equation}
	\label{eq:hatQ}
	\hat Q_i(n) \coloneqq \operatorname{span} \{\chi_i v_1, \dots, \chi_i v_{n} \},
\end{equation}
and the $(n+1)$-th eigenvalue characterizes the corresponding optimal \(n\)-width as
\begin{equation*}
	d_{n}(\omega_i, \omega_i^*) = \lambda_{n+1}^{-1/2}.
\end{equation*}
It can be shown that the $n$-width decays nearly exponentially in $n$, cf.~\cite{Ma24}. 
Due to the definition of $\mathcal{M}_i$, the eigenpairs of \cref{eq:evpharmonicfullincludingconstant} and \cref{eq:evpharmonicfull} are identical,  
except for the former including the zero eigenvalue corresponding to the constant eigenfunction, cf. \cite[Thm. 3.4]{Ma22}, and thus
$S_{n_i}(\omi) = \mathbb R \oplus \operatorname{span}\{v_1|_{\omi},\dots,v_{n_i-1}|_{\omi}\}$.
\subsection{MS-GFEM with eigenproblems defined on rings}
The proposed method differs from the MS-GFEM in the construction of the local approximation space. More precisely, the goal of the proposed method is to localize  eigenproblem \cref{eq:evpharmonicfull} to a ring. To formulate the localized eigenproblem, we again introduce two domains, namely a ring $\ringi \subset \omi$ chosen such that it contains the region where the subdomain~$\omega_i$ overlaps with its neighboring subdomains, and an oversampling ring $\ringosi$ with $\ringi \subset \ringosi \subset \osi$; see \cref{fig:patchandring} (right) for an illustration. 

Let $\ommini \subset \omi$ be a Lipschitz domain such that $\chi_i|_{\ommini} = 1$, $\ommini \cup \ringi = \omi$, and $\partial\ommini$ is 
contained in the interior of $\ringi$. {\cAdd We assume that}  
$\mathrm{dist}(\partial\ommini,\partial\ringi) \gtrsim \delta_i$. Then, there exists a cut-off function 
$\eta_i \in W^{1,\infty}(\osi, [0,1])$ that is $1$ on $\omi\setminus\ringi$ and $0$ on $\osi\setminus\ommini$. 
We define $\chi_i^R:\osi \to [0,1]$ such that $\chi_i^R(\bm x) = \chi_i(\bm x) - \eta_i(\bm x)$.
Then, there exists a constant~$C_{\chi^R}>0$ such that 
\begin{equation}
	\label{eq:propPUring}
	\max_{\bm x \in \osi}\,|\nabla \chi_i^R| \leq \frac{C_{\chi^R}}{{\cAdd \delta_i}}.
\end{equation}
The introduction of $\eta_i$ is inspired by the zero-extension 
used in the eigenproblems in \cite{Heinlein2019} and its purpose will become clear in \cref{thm:locapproxerr}.
\begin{figure}[h]
	\includegraphics{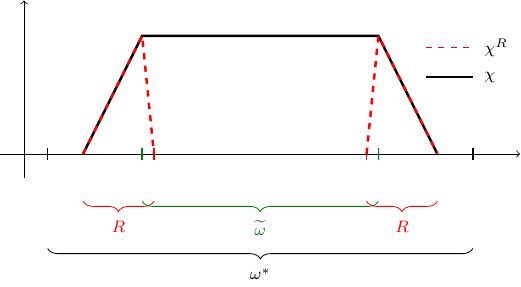}
	\caption{One dimensional illustration of the cut-off function~$\chi_i^R$ (left) and the corresponding ring $\ringi$ (right). }
	\label{fig:one_dimensional_illustration}
\end{figure}
	To build the local approximation space we consider the first $n_i$ eigenpairs
$(u_k,\lambda_k) \in H_{a}(R_i^*) \times \mathbb R$ of   
\begin{equation}
	\label{eq:evpharmonicringincludingconstant}
	a_{R_i^*}(u_k, \varphi) = \lambda_k\, a_{R_i}(\chi_i^R u_k, \chi_i^R \varphi)\qquad \forall\varphi \in H_{a}(R_i^*). 
\end{equation}
To obtain functions defined on the full oversampling domain $\osi$, we perform an operator-harmonic extension of the eigenfunctions. This extension is performed by taking the Dirichlet data on $\partial \ommini$. 
More precisely, the operator-harmonic extension of $u_k$, denoted by $u_k^\mathrm{ext}$, is defined on the subdomain $\omega_i^*$ as 
\begin{align}
	\label{eq:operatorharmonicextension}
	u_k^{\mathrm{ext}} \coloneqq \begin{cases}
		u_k&\quad \text{in } \ringosi \setminus \ommini,\\
		\tilde{u}_k &\quad \text{in } \ommini,
	\end{cases}
\end{align}
 where $\tilde{u}_k \in H^1(\ommini)$ satisfies the boundary condition $\tilde{u}_k = u_k$ on $\partial \ommini$ and 
 \begin{equation}
 	\label{eq:extensionring}
 	a_{\ommini}(\tilde{u}_k,\varphi) = 0\qquad \forall \varphi \in H^1_0(\ommini).
 \end{equation}
The operator-harmonically extended eigenfunctions are then used to construct the local approximation spaces of the proposed method. More specifically, the proposed method uses
\begin{align}
	\label{eq:locpartfunandlocappspace}
	u_i^p \coloneqq \psi_i|_{\omi},\qquad S_{n_i}(\omi) \coloneqq  \operatorname{span}\{u_1^\mathrm{ext}|_{\omi},\dots, u_{n_i}^\mathrm{ext}|_{\omi}\}.
\end{align}
For the analysis, we
consider another restriction operator
\begin{equation}
	\label{eq:compactRestrictionRing}
	P^R_i\colon  \aharmZ{\ringosi} \to H^1_{0}(\ringi), \qquad v \mapsto \chi^R_i v,
\end{equation}
whose domain is the space $H_{a,0}(\ringosi)$ of operator-harmonic functions on~$\ringosi$ for which  the functional 
\begin{equation}
	\label{eq:defMiR}
	\mathcal M_i^R\colon H^1(\osi)\to \mathbb R,\quad v \mapsto  \bigg(\int_{\omega_i} A \nabla \chi^R_i \cdot \nabla \chi^R_i \dx\bigg)^{-1}\int_{\omega_i} A \nabla (\chi_i^R v) \cdot \nabla \chi_i^R \dx
\end{equation}
vanishes.
Note that we have $v|_{R_i^*}\in H_{a,0}(R_i^*)$ for all $v\in H_{a}(\omega_i^*)$ with $\mathcal{M}_i^R(v) = 0$.
Since $\chi^R_i\in W^{1,\infty}(\osi)$, the compactness of $\P^R_i$ can be proved using the same arguments as in the corresponding proof for $P_i$.
Its Kolmogorov $n$-width $d_{n}(\ringi,\ringosi)$, defined analogously to \cref{eq:nwidthP}, is characterized by an eigenvalue problem similar to \cref{eq:evpharmonicfull} but posed on the ring~$\ringosi$. Specifically, it seeks eigenpairs $(v_k,\lambda_k) \in H_{a,0}(\ringosi)\times \mathbb R$ such that 
\begin{equation}\label{eq:evpharmonicring}
	a_{\ringosi}(v_k, \varphi) = \lambda_k\,a_{\ringi}(\chi^R_{i} v_k, \chi^R_{i} \varphi)\qquad \forall\varphi\in H_{a,0}(\ringosi). 
\end{equation}
	The operator-harmonic extension of $v_k$, denoted by $v_k^\mathrm{ext}$, is defined on the subdomain $\omega_i^*$ as in \cref{eq:operatorharmonicextension}.
Due to the definition of $\mathcal{M}_i^R$, the eigenpairs of \cref{eq:evpharmonicringincludingconstant} and \cref{eq:evpharmonicring} are identical, 
except for the former including the zero eigenvalue corresponding to the constant eigenfunction. Moreover
\begin{align}
	\label{eq:locpartfunandlocappspace}
	S_{n_i}(\omi) =  \mathbb R \oplus \operatorname{span}\{v_1^\mathrm{ext}|_{\omi},\dots, v_{n_i-1}^\mathrm{ext}|_{\omi}\}.
\end{align}

\subsection{Boundary subdomains}\label{sec:boundarysubdomains}
Until now, we have only considered interior subdomains. For boundary subdomains, the above choices of the local particular function and local approximation space must be adapted. Specifically, for the local particular function, we now choose $u_i^p \coloneqq (\psi_i+\psi_i^b)|_{\omi}$, where $\psi_i$ is given by \cref{eq:locparticularfun} and  $\psi_i^b \in H^1(\osi)$ is the unique function with $\psi_i^b = g$ on $\partial \osi \cap \partial \Omega$ satisfying
\begin{equation*}
	a_{\osi}(\psi_i^b,w) = 0\qquad \forall w \in \{v \in H^1(\osi) \with v|_{\partial \osi \cap \partial \Omega} = 0\}.
\end{equation*}
Note that, by construction, the particular function satisfies the Dirichlet boundary conditions on $\partial \osi\cap \partial \Omega$. 

Also the construction of the local approximation spaces must be adapted to the presence of boundary conditions. This can be done by considering the subspace of~$H_a(\ringosi)$ satisfying homogeneous Dirichlet boundary conditions on $\partial \ringosi \cap \partial \Omega$, i.e., 
\begin{equation}
	\label{eq:harmonic0boundary}
	H_{a,0}^b(\ringosi) \coloneqq \{v \in H_a(\ringosi) \with v|_{\partial \ringosi \cap \partial \Omega} = 0\}
\end{equation}
and solving eigenproblem~\cref{eq:evpharmonicring} in this space. 
The Dirichlet condition on $\partial\ringosi\cap\partial\Omega$ replaces the condition $\mathcal M_i v = 0$ in \cref{eq:harmonicfunctions0}.
The resulting eigenfunctions are then extended as in~\cref{eq:operatorharmonicextension}, and the local approximation space 
is defined as in~\cref{eq:locpartfunandlocappspace}, without adding $\R$. In summary, the following choices are made for 
boundary subdomains:
\begin{equation}
	\label{eq:locpartfunandlocappspaceboundary}
	u_i^p \coloneqq (\psi_i+\psi_i^b)|_{\omi},\qquad S_{n_i}(\omi) \coloneqq  \operatorname{span}\{v_1^\mathrm{ext}|_{\omi},\dots, v_{n_i}^\mathrm{ext}|_{\omi}\}.
\end{equation}

\section{Convergence theory}\label{sec:convergence}

In this section, we develop a convergence theory for the proposed method. The error analysis is divided into two parts. First, we study the approximation properties of the local approximation space computed only on the ring, and second, we combine the results to obtain an a priori error estimate of the proposed method. 

\subsection{Local approximation error}
Let us first examine the approximation properties of the $n$-dimensional subspace defined by
\begin{equation}
	\label{eq:qn}
	Q_i(n) \coloneqq \operatorname{span}\{\chi_i v_1^\mathrm{ext}, \dots, \chi_i v_{n}^\mathrm{ext}\}.
\end{equation}
when approximating the range of the restriction operator \(P_i\). Recall that the functions $v_i^\mathrm{ext}$ are constructed by solving eigenproblem~\cref{eq:evpharmonicring} on the ring $\ringosi$ and extending the resulting eigenfunctions as described in \cref{eq:operatorharmonicextension}. The following theorem bounds the corresponding approximation error  in terms of the \(n\)-width 
\begin{equation}
	\label{eq:nwidthP_ring}
	d_{n}(\ringi, \ringosi) \coloneqq \inf_{Q^R_i(n) \subset H_0^1(\ringi)} \sup_{u \in H_{a, 0}(\ringosi)} \inf_{v \in Q^R_i(n)} \frac{\|P^R u - v\|_{a, \ringi}}{\|u\|_{a, \ringosi}}.
\end{equation}

\begin{theorem}[Local approximation error bound]
	\label{thm:locapproxerr}
	The $n$-dimensional subspace $Q_i(n)$ defined in~\cref{eq:qn} satisfies
\begin{align}
	\label{eq:approxest}
	\inf_{v\in Q_i(n)}\anorm{ P u-v}{\omi} 
	\leq d_{n}(\ringi,\ringosi) \anorm{u}{\ringosi}
\end{align}
for all $u \in H_{a}(\osi)$ with $\mathcal{M}_i^R(u)=0$.
\end{theorem}
\begin{proof}
	To simplify the presentation of the proof, we consider a generic subdomain $\omega = \omega_i$ and omit the subscript~$i$. 
	First, we consider the case of an interior subdomain.
	Let $w = \sum_{k=1}^n \beta_k  v_k$ and $w^{\mathrm{ext}} = \sum_{k=1}^n \beta_k  v_k^\mathrm{ext}$ with $v_k$ and $v_k^\mathrm{ext}$ denoting the eigenfunctions in \cref{eq:evpharmonicring} and their harmonic extensions, cf. \cref{eq:operatorharmonicextension}.
	By construction we have $w|_{\ring\setminus\ommin} = w^{\mathrm{ext}}|_{\ring\setminus\ommin}$ and $\chi|_{\ring\setminus\ommin} = \chi^R|_{\ring\setminus\ommin}$, thus
	\begin{align}
		\label{eq:approximationResultRing}
		\anorm{\chi(u-w^{\mathrm{ext}})}{{\ring\setminus\ommin}} 
		= \anorm{\chi^R(u-w)}{{\ring\setminus\ommin}}. 
	\end{align}
	Next, we estimate the energy of $P u-v$ on the domain $\ommin$
	for an arbitrary $v=\chi w^{\mathrm{ext}}\in Q(n)$. 
	Since $u \in  H_{a}(\os)$ and $w^{\mathrm{ext}}|_{\ommin}$ is operator-harmonic on $\ommin$,
	\begin{align*}
		a_{\ommin}(u|_{\ommin} - w^{\mathrm{ext}}|_{\ommin}, \varphi) = 0\qquad \forall \varphi\in H_0^1(\ommin). 
	\end{align*}
	Since $w|_{\partial\ommin} = w^{\mathrm{ext}}|_{\partial\ommin}$ and $\chi^R|_{\partial\ommin} = 1$ the functions 
	$u-w^{\mathrm{ext}}$ and $\chi^R(u-w)$ have the same trace on $\partial\ommin$ and thus energy-minimality of $u-w^{\mathrm{ext}}$ implies
	\begin{align}
		\label{eq:approximationResultInner}
		\anorm{u-w^\mathrm{ext}}{\ommin}
		\leq \anorm{\chi^R(u-w)}{\ommin}
		= \anorm{\chi^R(u - w)}{\ommin\cap\ring}.
	\end{align} 
	Inequalities \cref{eq:approximationResultRing} and \cref{eq:approximationResultInner} together with 
	$\chi|_{\ommin} = 1$ imply that
	\begin{align*}
		\anorm{P u - v}{\om}^2 
		= \anorm{u-w^\mathrm{ext}}{\ommin}^2 + \anorm{\chi(u-w^{\mathrm{ext}})}{\ring\setminus\ommin}^2
		\leq \anorm{\chi^R (u - w)}{\ring}^2.
	\end{align*}
	The bound in \cref{eq:approxest} 
	follows from the definition of the $n$-width $d_{n}(\ringi,\ringosi)$ in \cref{eq:nwidthP_ring} 
	by taking the infimum over all $v\in Q(n)$, since $u \in H_{a}(\osi)$ and $\mathcal{M}_i^R(u)=0$ imply $u|_{\ringos}\in H_{a,0}(\ringos)$.
	For boundary subdomains, the result and its proof are very similar, except that $H_{a}(\os)$ and $H_{a,0}(\ringos)$ are replaced by $H_{a,0}^b(\os)$ and $H_{a,0}^b(\ringos)$ as defined in~\cref{eq:harmonic0boundary}, which satisfy homogeneous Dirichlet boundary conditions on $\partial \os \cap \partial \Omega$. The rest of the proof directly generalizes to boundary subdomains.
\end{proof}

\begin{remark}[Alternative method]
	\label{rem:cutoff}
	One could also consider $\ommini = \omi\setminus\ringi$ and define~$\chi_i^R$ by setting $\chi_i^R = \chi_i$ on $\ringi$ and 
	$\chi_i^R = 0$ on $\osi\setminus\ringi$.
    This defines an alternative method to the one proposed in this work. Numerical experiments show that 
	both approaches perform similarly.
	However, the theoretical analysis for the alternative method is more involved. The crucial point lies in estimating
	the approximation error in $\ommini$ by the approximation error 
	in $\ringi$. We achieved this by using energy-minimality in~\cref{eq:approximationResultInner}.
	This argument only works because $\ommini\cap\ringi\neq\emptyset$ and 
	$\chi_i^R$ is a cut-off function to the interior and hence
	allows for an extension by $0$ to the whole of~$\ommini$. For the alternative method this
	is not possible and instead of energy-minimality one has to use extension, trace and Poincaré inequalities, 
	which introduce additional dependencies on 
	the contrast and the width of $\ringi$. 
\end{remark}

The  theorem above estimated the local approximation error in terms of the \mbox{$n$-width} $d_{n}(\ringi,\ringosi)$. 
In the following theorem we establish the nearly exponential decay of this $n$-width with respect to the dimension of the 
approximation space. For simplicity, as in \cite{Ma22}, we restrict the detailed analysis to cubic subdomains~$\omega_i$. 
We call the subdomains $\ringosi$ and $\ringi$ concentric rings if there exist concentric (truncated) cubes $B_1, B_2, B_1^*,B_2^*\subset \Omega$ 
with $B_2\subset B_1$ and $B_2^*\subset B_1^*$ such~that
\begin{align*}
	\ringosi = B_1^* \setminus B_2^*, \quad \ringi = B_1 \setminus B_2.
\end{align*}
For a generalization to more general subdomains, see also \cite[Rem.~3.15]{Ma24}.
\begin{theorem}[Nearly exponential decay of $n$-width]\label{thm:decaykolmogorov}
Assume that $\ringi$ and $\ringosi$ are concentric rings.
There exist constants ${C_{d,i}, {\cAdd c_{d,i}>0}}$, and $n_0 \in \mathbb N$ such that it holds for any $n \geq n_0$ that
\begin{equation}\label{eq:1-5-5}
	d_{n}(\ring,\ringos)
		\leq  C_{d,i} e^{-{\cAdd c_{d,i}} n^{1/d}}.
\end{equation}
Denoting by $H^*_i$ the diameter of the subdomain~$\osi$, 
the constant $C_{d,i}$ scales as
\begin{equation*}
C_{d,i} \sim C_{\alpha, i}^{1/2} C_{\chi^R}\frac{H_i^{\ast}}{\cAdd{\delta_i}}.
\end{equation*} 
The constant $C_{\alpha,i}$ is the local contrast on $\ringosi$ defined as the local upper bound of the coefficient $\bm A$ divided by its local lower bound, cf. \cref{eq:propA}.

\end{theorem}

\begin{proof}
	The statement can be proved by slightly modifying the proof of \cite[Thm.~3.8~(i)]{Ma24}. For the sake of brevity,
	 we will only mention the main steps and restrict ourselves to boundary subdomains.

	 The main ingredients of the proof are a Caccioppoli inequality and a weak approximation property; see \cite[Lem.~5.5~(i) \& Eq.~(5.10)]{Ma24}. 
	 We replace the nested concentric cubes in the proof of \cite[Thm.~3.8~(i)]{Ma24} by $N+1$ nested concentric rings $\ringos = \ring^1 \supset\ldots\supset \ring^{N+1}=R^{N+1}\supset\ring$ 
	 with $\mathrm{dist}_{\infty}(\ring^{k+1}, \ring^k\setminus\partial\Omega) = \delta^* /(2N)$ and $\mathrm{dist}(\ring, R^{N+1}\setminus\partial\Omega) = \delta^* /2$
	 where $\delta^* = \mathrm{dist}_{\infty}(\ring, \partial\ringos\setminus\partial\Omega)$,
	 where $\mathrm{dist}_{\infty}$ is the distance with respect to the infinity norm. 
	 The starting point of the proof is a multiplicative decomposition of the operator $P^R$ 
	 \begin{equation*}
		P^R = X_a Q_2 Q_1,
	\end{equation*}
	into the operator $X_a: H_{a,0}^b(\ring) \to H^1_0(\ring), u\mapsto \chi^R u$ and the restriction operators $Q_1: H_{a,0}^b(\ringos) \to H_{a,0}^b(R^{N+1})$
	and $Q_2 : H_{a,0}^b(R^{N+1}) \to H_{a,0}^b(\ring)$.
	 To eliminate the operator norm of the operator $X_a$ in the 
	 statement of \cite[Thm.~3.8~(i)], which would result in an additional constant involving the width of the ring $R$, we 
	 modify the proof. 
	Specifically, instead of \cite[Eq.~(3.22)]{Ma24}, we use for all $m, k \in \mathbb N$ the estimate
	\begin{equation}
		\label{eq:multiplicative_decomposition_widths}
		d_{m + k}(\ring,\ringos)\leq d_{m}(X_a Q_2) d_{k}(Q_1),
	\end{equation} 
	where $d_{m}(Q_1)$ and $d_m(X_aQ_2)$ denote the corresponding \mbox{$m$-widths} of the operators $Q_1$ and $X_aQ_2$, respectively.
	Notice that in contrast to \cite{Ma24}, $R^{N+1}\neq \ringos$ is an intermediate subdomain. 
	
	First, we bound the term $d_{m}(X_aQ_2)$. By \cite[Lem.~5.5~(ii)]{Ma24}
	there exist constants $C_1$ and $C_2$, depending only on $d$, such that for each 
	$m\geq C_1|R^{N+1}| (\delta^*/2)^{-d}$, 
	there exists an $m$-dimensional subspace $Y_{m}(R^{N+1})\subset L^2(R^{N+1})$ such that 
	\begin{equation*}
		\inf_{v\in Y_{m}(R^{N+1})} \norm{u-v}{L^2(\ring)} \leq C_2 \norm{\nabla u}{L^{2}(R^{N+1})} |R^{N+1}|^{1/d} m^{-1/d} 
		\quad \forall u\in H^1(R^{N+1}).
	\end{equation*}
	We fix $m$ to be the smallest integer larger than $2^d C_1(\frac{H^*}{\delta^*})^d$.
	Since $Q_2$ considered as a operator from $H_{a,0}^b(R^{N+1})$ to $L^2(\ring)$ is compact, 
	there also exists a $m$-dimensional subspace $W_{m}(R^{N+1})\subset H_{a,0}^b(R^{N+1})$ such that
	\begin{equation*}
		\inf_{v\in W_{m}(R^{N+1})} \norm{u-v}{L^2(\ring)} \leq C_2 \norm{\nabla u}{L^{2}(R^{N+1})} |R^{N+1}|^{1/d} m^{-1/d}
		\quad \forall u\in H_{a,0}^b(R^{N+1}).
	\end{equation*}
	The previous result and the Caccioppoli inequality imply
	\begin{align*}
		\inf_{v\in W_{m}(R^{N+1})} \anorm{\chi^R(u-v)}{\ring}
		&\leq C_{\alpha}^{1/2}C_{\chi^R} \delta^{-1} C_2 \anorm{u}{R^{N+1}} |R^{N+1}|^{1/d} m^{-1/d}
	\end{align*}
	and thus
	\begin{align}
		\label{eq:proof_decaykolmogorov_XaQ}
		d_{m}(X_aQ_2) 
		\leq C_{\alpha}^{1/2} C_{\chi^R} \frac{H^*}{\delta}
		C_2  m^{-1/d}.
	\end{align}
	Using the nested concentric rings $\ring^1,\ldots, \ring^{N+1}$ and proceeding as in the proof of \cite[Thm.~3.8~(i)]{Ma24}, we obtain that 
	\begin{equation}
	\label{eq:proof_decaykolmogorov_ntilde}
		d_{k}(Q_1) \leq e^{-\widetilde{c}k^{1/d}}
	\end{equation}
	for all $k \geq k_0 = (e\Theta)^d$, where 
		$$\Theta = 4 (3d)^{1/d} C_2 C_{\alpha}^{1/2} 
		\left(\frac{H^*}{\delta^*}\right)^{1-1/d},\qquad \widetilde{c} = (e\Theta + 1)^{-1}.$$
		Notice that the constant $\Theta$ is slightly different to the one for cubes 
		for three reasons. First, only a part of the oversampling is used for the exponential bound 
		in~\cref{eq:proof_decaykolmogorov_ntilde}, because we introduced the intermediate subdomain 
		$R^{N+1}$. Second, $\delta^* = \mathrm{dist}_{\infty}(\ring, \partial\ringos\setminus\partial\Omega)$ 
		instead of $2 \mathrm{dist}_{\infty}(\om, \partial\os\setminus\partial\Omega) = H^*-H$ in \cite{Ma24}.
		Third, one gets an additional factor of $2$ when bounding the measure of the sets 
		$V_{\delta^*/(4N)}(\ring^k\setminus\ring^{k+1})$ 
		, which 
		now consists of two rings, cf. 
		\cite[Lem. 3.13]{Ma24}.	Let $n= k + m$ and $n_0:= k_0 + 2 m$. 
		Since $k^{1/d} = (n-m)^{1/d} \geq 2^{-1/d}n^{1/d}$, \cref{eq:multiplicative_decomposition_widths}, \cref{eq:proof_decaykolmogorov_XaQ}
		and \cref{eq:proof_decaykolmogorov_ntilde} yield
	\begin{equation*}
		d_{n}(\ring,\ringos)
		\leq C_{\alpha}^{1/2} C_{\chi^R} \frac{H^*}{\delta}
		C_2 e^{-\widetilde{c} k^{1/d}}
		\leq C_{\alpha}^{1/2} C_{\chi^R} \frac{H^*}{\delta}
		C_2 e^{-2^{-1/d} \widetilde{c} n^{1/d}}.\qedhere
	\end{equation*}
\end{proof}

\subsection{A priori convergence}
After having obtained estimates for the local approximation error, we now use them in \cref{thm:1-0} to derive a global error estimate for the solution of the proposed method, formalized in the following theorem.
	\begin{theorem}[Error estimate]\label{thm:errorestimate}
		The approximation of the proposed method, satisfies the global error estimate
		\begin{align}\label{eq:errorgfemring}
				\|u-u^G\|_a &\leq \sqrt{\kappa\kappa^*} d_{\bm n}\|u\|_a,
		\end{align}
		where $d_{\bm n}$ is defined as
		\begin{equation*}
			d_{\bm n} \coloneqq \max_{i = 1,\dots,M}\; \left\{\begin{array}{rl}
				d_{n_i-1}(\ringi,\ringosi) & \text{ if }\partial \osi\cap \partial \Omega \neq \emptyset,\\
				d_{n_i}(\ringi,\ringosi)& \text{ else.}		\end{array}\right.
		\end{equation*}
	\end{theorem}

\begin{proof}	
	It suffices to bound the local approximation error~\cref{eq:1-3-1}. 
	The result then follows directly from \cref{thm:1-0}. We start with interior subdomains; boundary 
	subdomains will be discussed later. 
	
	To simplify the presentation of the proof, we will again omit the 
	subscript~$i$ for the subdomains. First, we note that the definition 
	of the local particular function in \cref{eq:locparticularfun} and the relation $H^1_0(\os) \subset H^1_0(\Omega)$ 
	(implicit extension by zero) imply that
	\begin{equation}\label{eq:1-5-4-2}
		a_{\os}(u|_{\os}-\psi,\varphi) = 0\qquad \forall  \varphi\in H_{0}^{1}(\omega^{\ast}_{i}).
	\end{equation}
We now construct a function $\phi \in S_{n}(\om)$ for which the local approximation error~\cref{eq:1-3-1} is small. 
In particular, we choose $v \in V_{n-1} \coloneqq \operatorname{span}\{v_1^\mathrm{ext}|_{\om},\dots,v_{n-1}^\mathrm{ext}|_{\om}\}$
and set $\phi = \mathcal M^R(u-\psi)+v$ 
such that due to~\cref{eq:locpartfunandlocappspace}, the local approximation error is
	\begin{align*}
\|\chi(u|_{\os} - \psi - \phi)\|_{a,\om}
= \|\chi(u|_{\os}-\psi - \mathcal M^R(u|_{\os}-\psi) -v)\|_{a,\om}.
	\end{align*}
	Since $u|_{\om} - \psi - \mathcal M^R(u|_{\om}-\psi) \in H_{a}(\os)$, \cref{thm:locapproxerr} guarantees 
	the existence of $v \in V_{n-1}$ such that 
	\begin{align}
		\label{eq:locestinteriorsubdomain}
		\begin{split}
			\|\chi(u|_{\os} - \psi - \phi)\|_{a,\om}
			&\leq d_{n-1}(\ring,\ringos)\|u|_{\os} - \psi - \mathcal M^R(u|_{\os}-\psi)\|_{a,\ringos}\\
			&\leq d_{n-1}(\ring,\ringos)\|u\|_{a,\os}.
		\end{split}
	\end{align}
Here we used that due to \cref{eq:1-5-4-2}, $\|u|_{\os}-\psi\|_{a,\os} \leq \|u\|_{a,\os}$.
	
For boundary subdomains, the proof is very similar. The main difference is the use of a different local 
particular function, cf.~\cref{eq:locpartfunandlocappspaceboundary}, as well as a different space of operator-harmonic functions, 
where Dirichlet boundary conditions are imposed, cf.~\cref{eq:harmonic0boundary}. 
Due to the Dirichlet boundary conditions, $S_n(\om)$ has no additional constant component, cf. \cref{eq:locpartfunandlocappspaceboundary}, and we 
obtain the estimate
\begin{align*}
	\|\chi(u|_{\os} - \psi - \psi^b - \phi)\|_{a,\om}
	\leq d_{n}(\ring,\ringos)\|u\|_{a,\os},
\end{align*}
using $d_{n-1}(\ring,\ringos)$ instead of $d_{n}(\ring,\ringos)$ in \cref{eq:locestinteriorsubdomain}.
\end{proof}

Substituting the a priori decay estimate for the $n$-width $d_n(\ringi,\ringosi)$ from \cref{thm:decaykolmogorov} into the estimate from the previous theorem yields an a priori error estimate for the proposed method, which is given in the following corollary.
\begin{corollary}[A priori error estimate]\label{cor:apriorierrorstimate}
	The approximation of the proposed method satisfies the a priori error estimate
	\begin{equation}
	\|u-u^G\|_a	\leq  C_d\exp(c_d) \sqrt{\kappa\kappa^*} \exp(-c_dn^{{1}/{d}}),
	\end{equation}
	 where we set $n \coloneqq  n_1= \dots = n_M$, $C_d \coloneqq \max\limits_{i = 1,\dots,M} C_{d,i}$ 
	 and $c_d \coloneqq \min\limits_{i = 1,\dots,M} c_{d,i}$.
\end{corollary}

	We emphasize that in order to obtain a practical method, the local but still infinite-dimensional problems \cref{eq:locparticularfun,eq:extensionring,eq:evpharmonicring} must be discretized. For this purpose, classical finite element methods are typically used. 
	For an analysis of the resulting fully discrete MS-GFEM, we refer to \cite{Ma24}. It can be adapted easily to the present setting of rings.
More specifically, the main tools in \cref{thm:locapproxerr} and \cref{thm:decaykolmogorov}
were 
energy-minimality of operator-harmonic functions, the Caccioppoli inequality and the weak approximation 
property. All of these properties have discrete analogues and hence 
the modification of MS-GFEM to use eigenproblems on the rings does not affect the analysis in the fully 
discrete setting.

\section{Application as preconditioner}\label{sec:preconditioner}

As shown in \cite{Strehlow2024}, the proposed multiscale method can also be written as an iterative method, which motivates its use
as a two-level Restricted Additive Schwarz (RAS) preconditioner. 
The presentation of this section closely follows \cite{Strehlow2024}.

\subsection{Iterative version}\label{iterativemsgfem}
To derive an iterative version of the proposed method, we  interpret it within the framework of domain decomposition methods. 
Specifically, the particular functions $u^p_i$ in \cref{eq:locparticularfun} can be viewed as local subdomain solves on the oversampling
 domains \(\{\osi\}_{i=1}^M\), where the results are ``glued'' together using the partition of unity \(\{\chi_i\}_{i=1}^M\). 
 Furthermore, the solution of the proposed method itself can be understood as a coarse space correction applied to this one-level
  method, with \(S_{\bm n}(\Omega)\) in \cref{eq:locpartfunandlocappspace} and \cref{eq:locpartfunandlocappspaceboundary} serving as the coarse space. This becomes clear when writing the approximation as 
  \(u^G = u^s + u^p\), where \(u^s \in S_{\bm n}(\Omega)\) is the unique function satisfying
\[
a(u^s, v) = a(u - u^p, v)\qquad \forall  v \in S_{\bm n}(\Omega), 
\]
which is a reformulation of~\cref{eq:weakformulationgfem}.
These two steps, namely the local subdomain solves and the coarse grid correction, define one iteration of the iterative method. 
In subsequent iterations, we use the approximate solution $u_G$ as Dirichlet data of the previous iteration for the local subdomain solves and again perform a coarse grid correction. Repeating this process yields the desired iterative method.

To formalize this iteration, we introduce the linear operator $G$, which maps a function $w \in H^1(\Omega)$ to the solution $u^G$ of the proposed method solving~\cref{eq:weakformulationgfem} for the right-hand side $F\coloneqq a(w,\cdot)$. This mapping can be equivalently expressed as
\begin{equation}
	\label{eq:edfG}
	Gw = \sum_{i=1}^M \chi_i \pi_i(w) + \pi_s \bigg( w - \sum_{i=1}^M \chi_i \pi_i(w) \bigg)
\end{equation}
with $\pi_i\colon H^1(\Omega)\to H^1_0(\osi)$ and $\pi_s\colon H^1(\Omega)\to S_{\bm n}(\Omega)$ denoting the $a$-orthogonal projections onto $H^1_0(\osi)$ and $S_{\bm n}(\Omega)$, respectively.
Given an initial guess ${u^{0} \in H^1(\Omega)}$, the sequence of iterates $\{u^{j}\}_{j =0}^\infty$ is then defined for $j = 0,1,\dots$ by the recurrence
\begin{equation}
	\label{eq:iteration}
	u^{j+1} \coloneqq  u^{j} + G\big( u - u^{j} \big).
\end{equation}
Note that this iteration can be viewed as a preconditioned Richardson-type iteration, where~$G$ is the preconditioned operator. 
The following lemma proves the strictly monotone convergence of this iteration with respect to the energy norm,
provided $\bm n$ is sufficiently large such that $G$ is a contraction in the energy norm.
 
\begin{lemma}[Convergence of iteration] \label{cor2}
	Let $\bm n = (n_1,\dots,n_M)$ be chosen such that
	\begin{equation}
		\label{eq:assumptionexistencerate}
		\|u-u^G\|_a \leq \vartheta \|u\|_a
	\end{equation}
	holds for a constant $0<\vartheta<1$. Then the iteration $\{v^j\}_{j =0}^\infty$ defined by \cref{eq:iteration} converges for any initial guess $u^{0}\in H^1(\Omega)$ and   it holds for all $j = 0,1,\dots$ that
\begin{equation}
	\label{eq:convergenceiteration}
	\lVert u^{j+1} - u \rVert_a \leq \vartheta \lVert u^{j} - u \rVert_{a}.
\end{equation}
\end{lemma}
\begin{proof}
It follows directly from \cref{eq:assumptionexistencerate} that the inequality \( \|(I - G)v\|_a \leq \vartheta \|v\|_a \) holds for all \( v \in H^1(\Omega) \), where $I$ denotes the identity operator. Subtracting the solution~\( u \) of \cref{eq:PDEweak} from both sides of equation \cref{eq:iteration}, we obtain that
\[
u^{j+1} - u \coloneqq (I - G)(u^{j} - u).
\]
Applying the above estimate for the operator \( G \), the assertion can be concluded.
\end{proof}

\subsection{Preconditioned GMRES}
To construct a preconditioner for GMRES, we first introduce a fine-scale finite element discretization and give the matrix representations of the corresponding discrete operators. We consider a fine mesh~$\mathcal T_h$, which resolves all microscopic details of the coefficients, and denote by \( V^h \subset H^1(\Omega) \) a corresponding classical polynomial-based finite element space. We further introduce the subspace \( V_0^h \coloneqq V^h \cap H^1_0(\Omega) \) of functions satisfying homogeneous Dirichlet boundary conditions on $\partial \Omega$. Given a subdomain \( S \subset \Omega \), we also introduce local versions of the spaces \( V^h \) and \( V_0^h \) by \( V^h(S) \coloneqq (V^h)|_S \) and \( V_0^h(S) \coloneqq (V^h \cap H^1_0(S))|_S \), respectively. Note that for the matrix representations given below we will always consider a classical Lagrange finite element basis.

The matrix representation of the extension-by-zero operator $V_0^h(\omega_i^*) \rightarrow V^h$ is in the following denoted by ${\bm R}_i^T$. Similarly, we denote by ${\bm R}^{T}_S$ the matrix representation of the natural embedding $S_{\bm n}^h(\Omega) \rightarrow V^h$, where $S_{\bm n}^h(\Omega) \subset V_0^h$  is a fully discrete version of the GFEM space $S_{\bm n}(\Omega)$, cf.~\cref{eq:Sn}.  Denoting by $\bm{K}$ the stiffness matrix associated with the bilinear form $a$, we define $\bm{K}_i \coloneqq {\bm R}_i \bm{K} {\bm R}_i^T$ and $\bm{K}_S \coloneqq {\bm R}_S \bm{K} {\bm R}_S^T$. 
The projections $\pi_i$ and $\pi_S$ can then be written in matrix form as:
\begin{equation*}
	\boldsymbol{\pi}_i = \bm{K}_i^{-1} {\bm R}_i \bm{K},\qquad \boldsymbol{\pi}_S = \bm{K}_S ^{-1}{\bm R}_S \bm{K}.
\end{equation*}
The matrix form of the linear operator \( G \) then reads
\begin{equation*}
	{\bm G} = {\bm B \bm{K}}, \quad
	{\bm B} \coloneqq  \bigg( \sum_{i=1}^{M} {\bm R}_i^T {\boldsymbol{\chi}}_i \bm{K}_i^{-1} {\bm R}_i \bigg) + 
	\big({\bm R}_S^T \bm{K}_S^{-1} {\bm R}_S \big) 
	\bigg( {\bm I} - \sum_{i=1}^{M} {\bm R}_i^T {\boldsymbol{\chi}}_i \bm{K}_i^{-1} {\bm R}_i \bigg),
\end{equation*}
where \( {\bm I} \) denotes the identity matrix and \( {\boldsymbol{\chi}}_i \) is the matrix representation 
of the {\cAdd operator \( V^h_0(\omega_i^*) \to V^h_0(\omega_i^*) \), \( v^h \mapsto I_h(\chi_i v^h) \) with the nodal interpolation operator~\( I_h \).}

Thus, the preconditioned system of equations can be obtained by multiplying $\bm{K}\bm{u} = \bm{f}$ from the left with $\bm{B}$, i.e., 
\begin{equation} \label{preconditioned-system}
	{\bm B \bm{K} \bm u} = {\bm B \bm f}.
\end{equation}
The preconditioner ${\bm B}$ can be viewed as a Two-Level Hybrid Restricted Additive Schwarz preconditioner, where the word ``hybrid'' refers to the multiplicative incorporation of the coarse space. A preconditioner is typically used in conjunction with a Krylov subspace method to accelerate convergence. Since the matrix ${\bm B}$ is non-symmetric, we employ the GMRES method for solving~\cref{preconditioned-system}.

In each iteration, the GMRES method solves a least squares problem with respect to a given {\cAdd norm $\|\cdot\|_b$, which is induced by an inner product $b(\cdot,\cdot)$ on $\R^n$.} More specifically, starting with an initial guess $\bm{v}^{0} \in \mathbb{R}^n$, the sequence of GMRES iterates $\{\bm{v}^{j}\}_{j=0}^\infty$ is generated by solving the following minimization problem:
\begin{equation*}
	\label{eq:minimizationGMRES}
	\min_{\bm{v} \in \bm{v}^{0} + \mathcal{K}_j} \| \bm{B}\bm{K}(\bm{u} - \bm{v}) \|_b,
\end{equation*}
where $\mathcal{K}_j$ denotes the Krylov subspace of dimension $j \in \mathbb{N}$, defined as
\[
\mathcal{K}_j := \operatorname{span} \left\{ \bm{B}\bm{K}(\bm{u} - \bm{v}^{0}), (\bm{B}\bm{K})^2(\bm{u} - \bm{v}^{0}), \dots, (\bm{B}\bm{K})^j(\bm{u} - \bm{v}^{0}) \right\}.
\]
The minimization property implies that, when starting from the same initial guess, the GMRES residuals are less than or equal to the residuals of the preconditioned Richardson iteration, which is formalized in the following proposition.
{\cAdd
\begin{proposition}[Relation of residuals]
For all $j = 0,1, \dots$ it holds that
\[
\| \bm{B}\bm{K}(\bm{u} - \bm{v}^{j}) \|_b \leq \| \bm{B}\bm{K}(\bm{u} - {\cAdd \bm{u}^{j}}) \|_b,
\]
where $\{\bm{u}^{j}\}_{j = 0}^\infty$ denote the coefficient vectors of a discrete version of~\cref{eq:iteration} with the initial guess $\bm{u}^{0} = \bm{v}^{0}$. 
\end{proposition}
\begin{proof}
The proof of this result can be found in \cite[Prop.~3.5]{Strehlow2024}.
\end{proof}
}
The convergence properties of the GMRES method depend on the norm equivalence constants of the norm $\|\cdot\|_b$ and the energy vector norm, defined for all \( \bm{v} \in \mathbb{R}^n \) as \( \|\cdot\|_a^2 \coloneqq \bm{v}^T \bm{K} \bm{v} \). Specifically, 
it depends on the constants \( 0 < \beta_{\mathrm{min}} \leq \beta_{\mathrm{max}} < \infty \) 
in
\begin{equation*}
	\beta_{\mathrm{min}} \|\bm{v}\|_b \leq \|\bm{v}\|_a \leq \beta_{\mathrm{max}} \|\bm{v}\|_b\qquad \forall \bm{v} \in \mathbb R^n.
\end{equation*}
Note that if the considered family of meshes \(\{ \mathcal{T}_h \}_h\) is quasi-uniform in the sense of~\cite[Def.~1.140]{ErG04} {\cAdd and $b(\cdot,\cdot)$ is the Euclidean norm}, the norm equivalence constants scale as \( \beta_{\mathrm{min}} \sim \alpha_{\mathrm{max}}^{-1/2} h^{-d/2+1} \) and \( \beta_{\mathrm{max}} \sim \alpha_{\mathrm{min}}^{-1/2} h^{-d/2} \). 
\begin{theorem}[Convergence of preconditioned GMRES]
The GMRES iterates $\{\bm{v}^j\}_{j = 0}^\infty$ converge for any initial guess $\bm{v}^0 \in \mathbb R^n$ with the estimate
\[
\|\bm{B} \bm{K} (\bm{u} - \bm{v}^{j})\|_b \leq \vartheta^j \left( \frac{1 + \vartheta}{1 - \vartheta} \right) \frac{\beta_{\mathrm{max}}}{\beta_{\mathrm{min}}} \|\bm{B} \bm{K} (\bm{u} - \bm{v}^{0})\|_b,
\]
which holds for all \( j \in\N \).
\end{theorem}
\begin{proof}
See~\cite[Thm.~3.9]{Strehlow2024} for a proof.
\end{proof}

\section{Numerical experiments}\label{sec:numericalexperiments}

\begin{figure}
	\includegraphics[height = .4\linewidth, width=.5\linewidth]{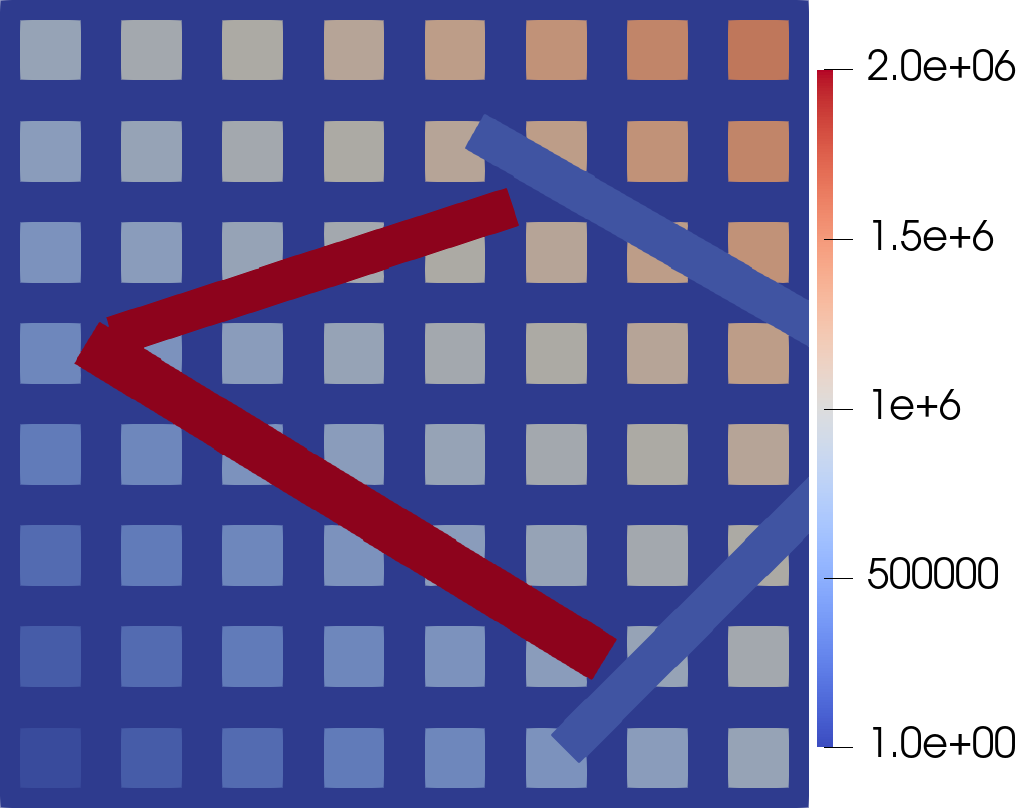}\hspace{.75cm}
	\includegraphics[height = .4\linewidth,width=.4\linewidth]{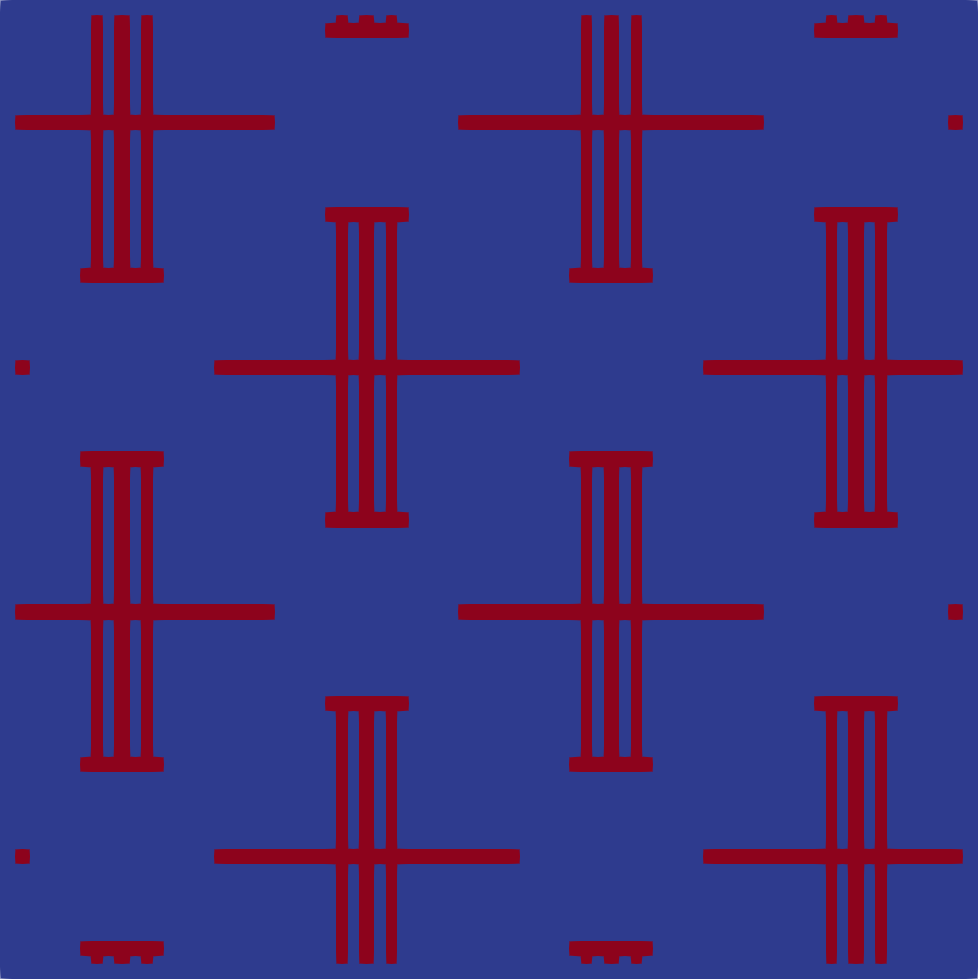}
	\caption{Coefficients $\bm{A}$ used in the numerical experiments.}
	\label{fig:coefficients}
\end{figure}

In this section, we describe the implementation of the proposed method on rings—referred to as MS-GFEM-R—and provide numerical experiments that support the theoretical results and offer a comparison with the original MS-GFEM, in which the local eigenproblems are defined on the entire subdomains. All experiments use a $\mathcal Q^1$-finite element method on uniform Cartesian meshes for the fine-scale discretization. We always choose the fine mesh size small enough to resolve the coefficient. 
The computational domain is divided into $M$ non-overlapping brick-like subdomains. Each subdomain is extended by two layers of fine elements in each space direction to create the overlapping subdomains $\{\omi\}_{i = 1}^m$ of~$\Omega$. 
Additionally, each subdomain $\omi$ is extended by $\ell$ layers of fine elements in each space direction to obtain the oversampling domain $\osi$. 
{\cAdd 
The oversampling rings $\ringosi$ are defined by extending $\supp (\chi_i)$ by $\ell$ layers of fine elements in each space direction.
The functions $\eta_i$, used to define $\chi_i^R$ are implemented as cut-off functions that transition linearly from $1$ to $0$ over one layer of fine elements
and we set $\ringi = \supp (\chi_i^R)$.}
In the following, for simplicity, the number of basis functions per subdomain is chosen to be the same for all subdomains, i.e., ${n = n_1=\dots=n_M}$. 
When used as a multiscale method, we use the relative energy error as a benchmark criterion, defined~as
\begin{equation}
	\mathbf{err} := \frac{\Vert u_{h} - u_{h}^{G}\Vert_{a}}{\Vert u_{h}\Vert_{a}},
\end{equation}
where we denote the fine-scale reference solution of the PDE by $u_h$, cf.~\cref{eq:PDEweak}, and the fine-scale discretized GFEM solution by $u_h^G$, cf.~\cref{eq:weakformulationgfem}.
When the methods are used as preconditioners, we use the number of iterations required to achieve a relative residual of $10^{-8}$ as a benchmark criterion. If the iteration count exceeds the maximal number of iterations of 1000, it is formally set to infinity.

The numerical experiments presented below can be reproduced using the code available at \url{https://github.com/ChristianAlber/msgfemr}.

\subsection{Nearly exponential convergence}

\begin{figure}
	\includegraphics[width=.45\linewidth]{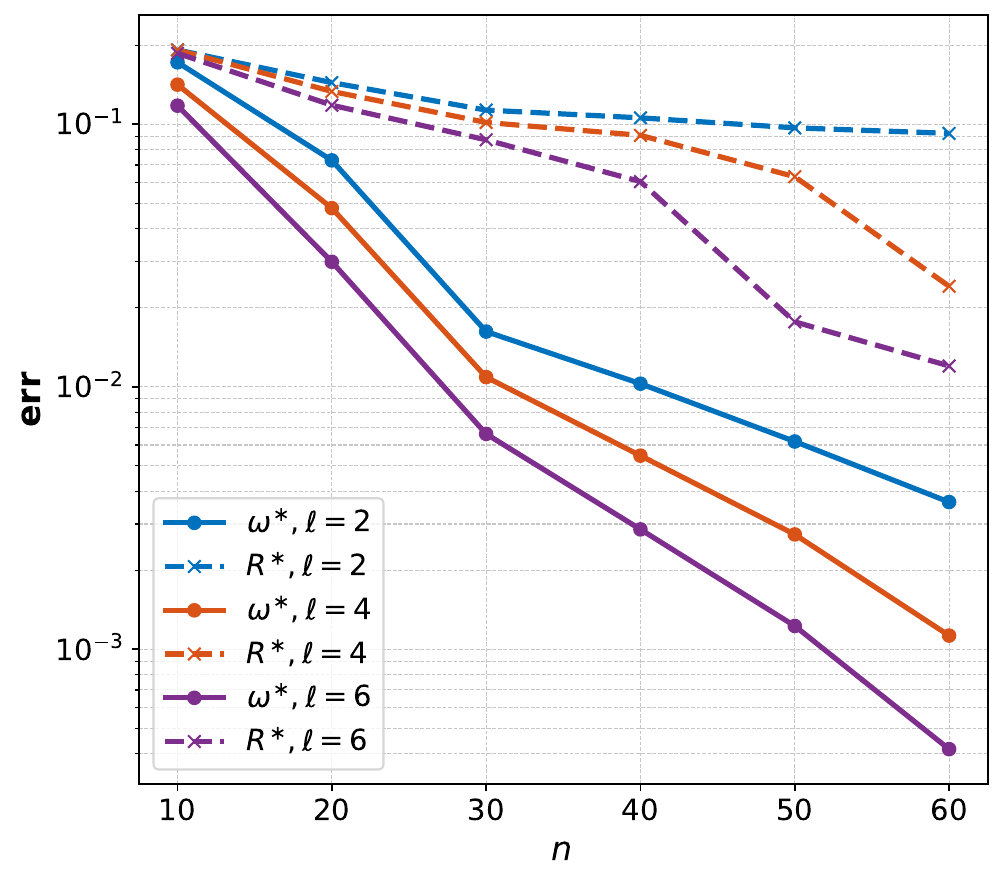}\hspace{.5cm}
	\includegraphics[width=.45\linewidth]{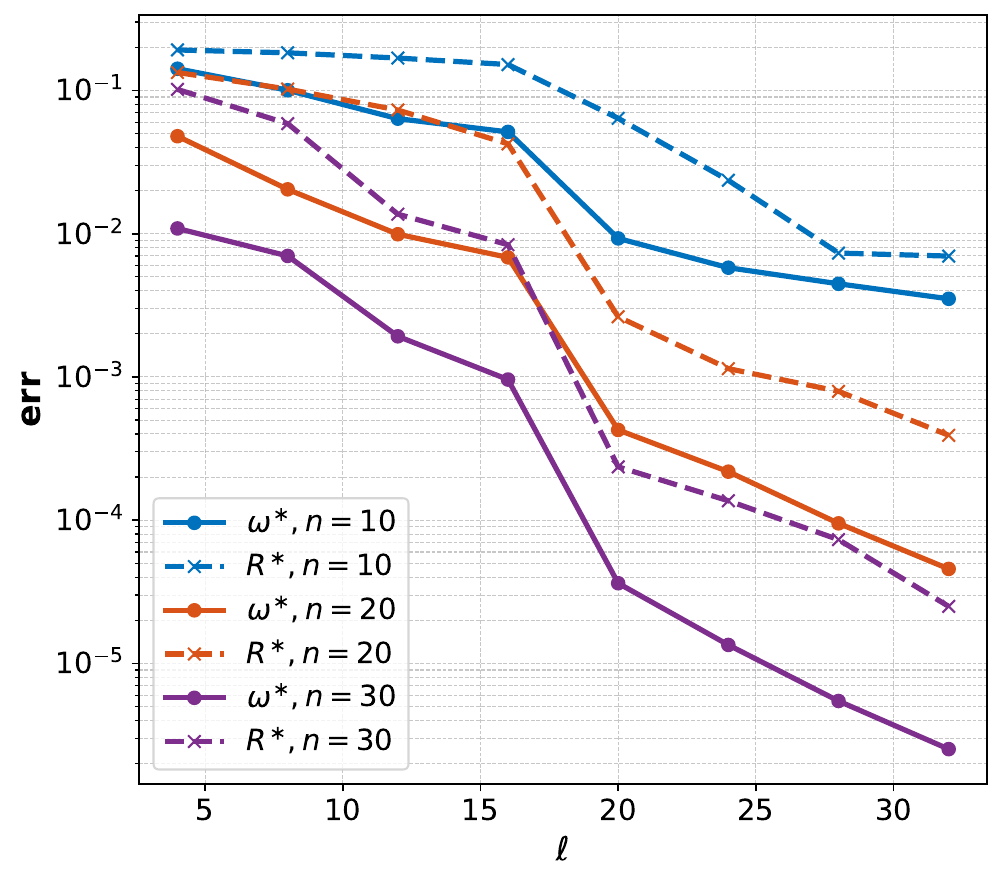}
	\caption{Errors with MS-GFEM and MS-GFEM-R for the skyscraper coefficient in \cref{fig:coefficients}(left) as a function of the number of local eigenfunctions $n$ (left) and as a function of the number of oversampling layers $\ell$ (right). The error curve of the MS-GFEM is marked with $\os$, while that of MS-GFEM-R is marked with $\ringos$.}
	\label{fig:subexponential_convergence}
\end{figure}

The first numerical experiment considers the domain $\Omega=(0,1)^2$, the so-called skyscraper coefficient $\bm A$ from \cite{ma2022error} shown in \cref{fig:coefficients}~(left), which has a contrast of $2 \times 10^6$. Furthermore, we consider the source term $f \equiv 1$ and homogeneous Dirichlet boundary conditions are imposed on~$\partial \Omega$. 
A Cartesian mesh with 800 elements in each space dimension is used for the fine-scale discretization. The number of subdomains is chosen to be $M=64$, which means that we have 8 subdomains in each space direction.  

\cref{fig:subexponential_convergence} (left) illustrates the (nearly-)exponential convergence of MS-GFEM and MS-GFEM-R,
 as the number of local eigenfunctions is increased, when used as multiscale methods.
 This numerically supports the assertions of \cref{thm:errorestimate,cor:apriorierrorstimate}. 
 One observes that roughly twice as many local eigenfunctions are needed to 
 achieve the same accuracy with MS-GFEM-R as with MS-GFEM. The heuristic explanation for this observation is that 
 (after a possible preasymptotic regime) the dimension of the space of (discrete) operator-harmonic functions on the ring 
 is about twice that of the (discrete) operator-harmonic space on the whole oversampling domain.  
 This is because an (discrete) operator-harmonic function is uniquely defined by its boundary values, 
 and the ring has about twice as many fine-scale degrees of freedom on the boundary than the original oversampling domain. 
Furthermore, in \cref{fig:subexponential_convergence} (right), one observes the (nearly-)exponential convergence of the MS-GFEM as the number of oversampling layers is increased. We have not explicitly proved this result, but such a result is known from \cite{Ma22} and the proof therein applies also to the present setting of rings.

\subsection{Convergence of preconditioned iteration}

\begin{figure}
	\includegraphics[width=\linewidth]{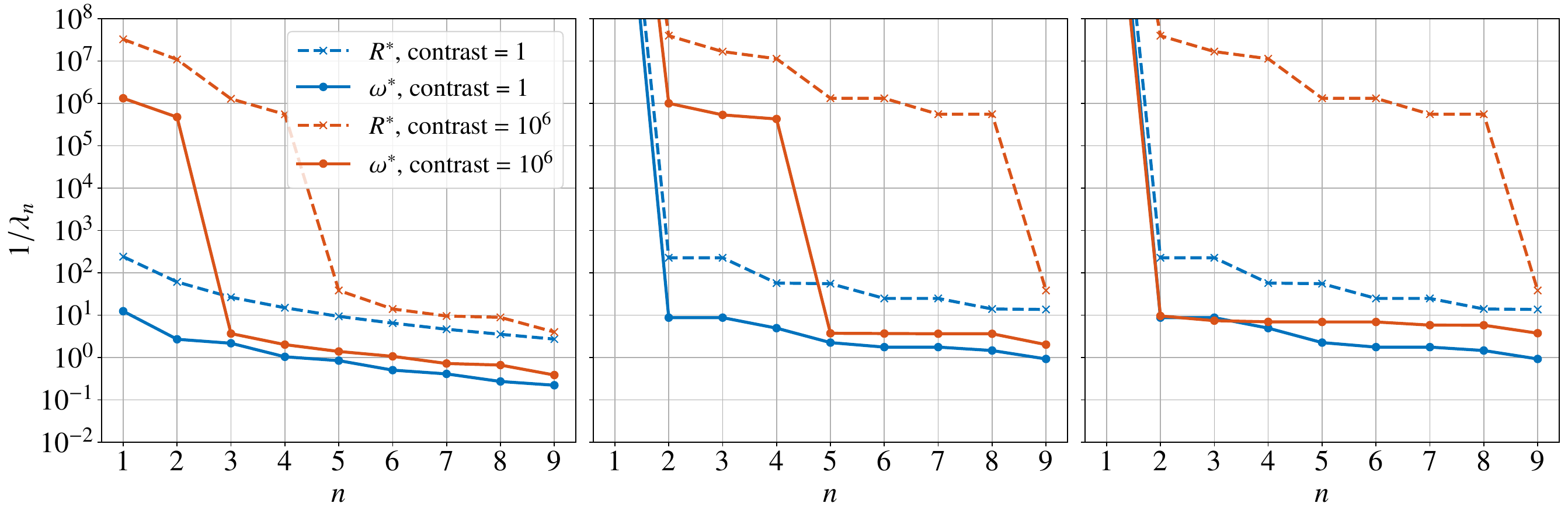}\\[2ex]
	\hspace{.4cm}
	\includegraphics[width = 0.3\linewidth]{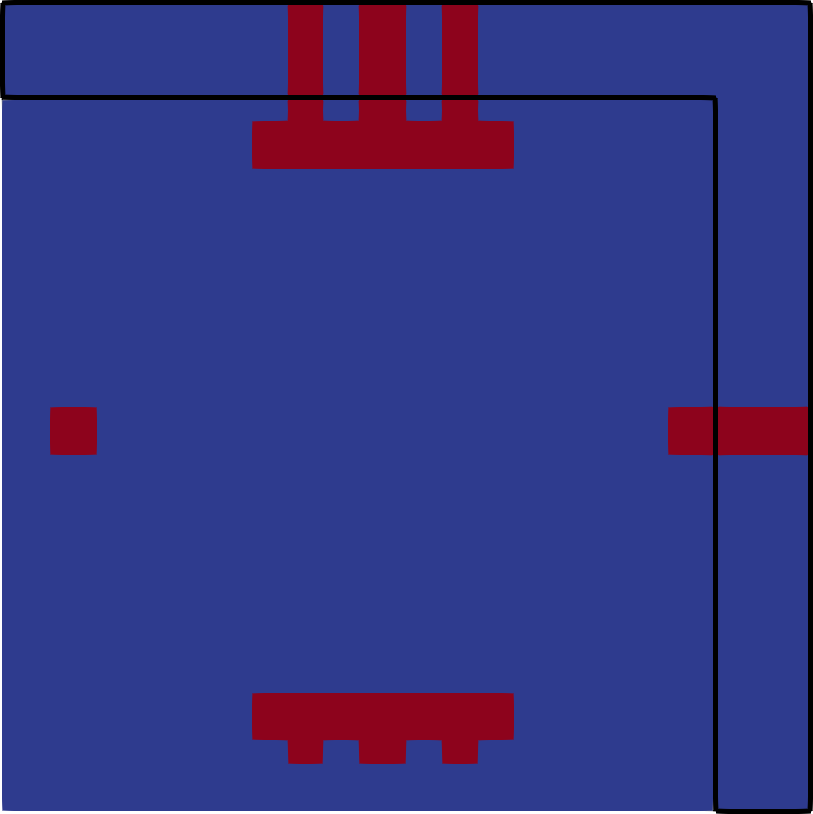}
	\hspace{.1cm}
	\includegraphics[width = 0.3\linewidth]{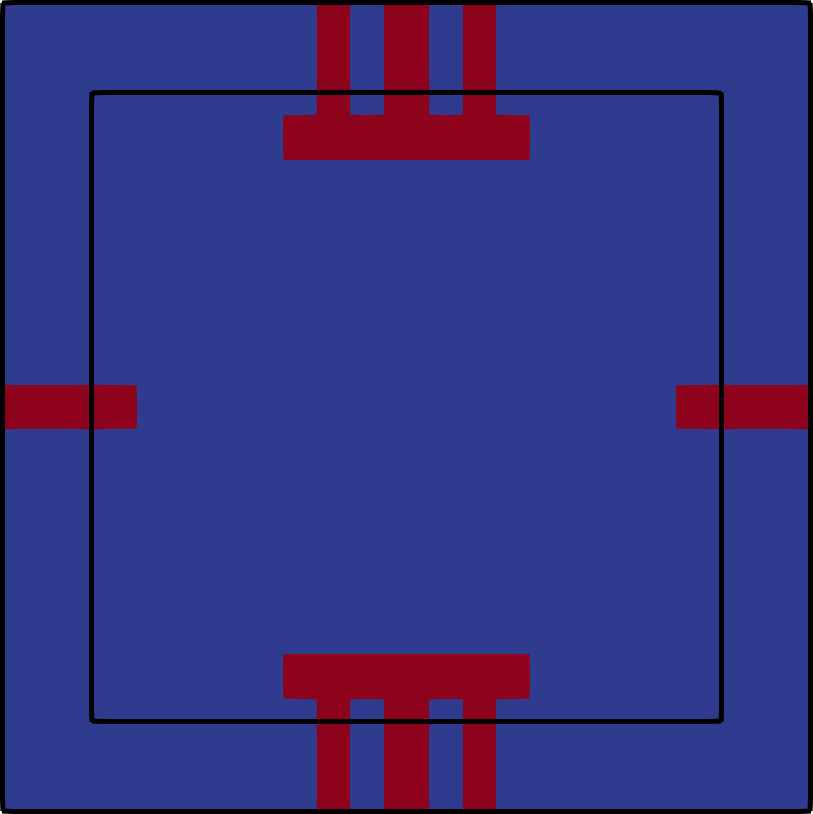}
	\hspace{.1cm}
	\includegraphics[width = 0.3\linewidth]{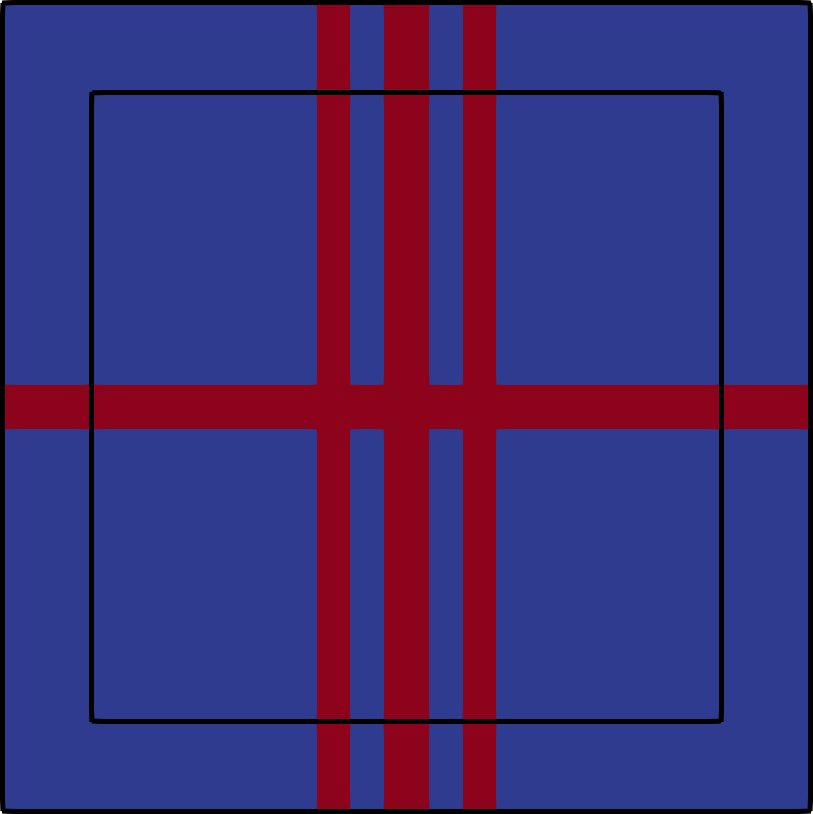}
	\caption{Eigenvalues (top row) of MS-GFEM and \mbox{MS-GFEM-R} for the channelized coefficient in \cref{fig:coefficients} (right)
	for three representative subdomains (bottom row).}
	\label{fig:channel_eigenvalues}
\end{figure}

\begin{table}
	\begin{subtable}{\linewidth}
		\begin{tabularx}{\textwidth}{l|XXXXXXXXXX}
\toprule
$\text{contrast} \setminus n$ & 1 & 2 & 3 & 4 & 5 & 6 & 7 & 8 & 9 & 10 \\
\midrule
$10^0$ & 90 & 75 & 37 & 25 & 23 & 21 & 18 & 13 & 13 & 11 \\
$10^3$ & $\infty$ & $\infty$ & 406 & 46 & 41 & 40 & 36 & 28 & 26 & 26 \\
$10^6$ & $\infty$ & $\infty$ & $\infty$ & 44 & 36 & 35 & 33 & 28 & 26 & 27 \\
\bottomrule
\end{tabularx}

	\end{subtable}\vspace{3ex}
	\begin{subtable}{\linewidth}
		\begin{tabularx}{\textwidth}{l|XXXXXXXXXX}
\toprule
$\text{contrast} \setminus n$ & 1 & 2 & 3 & 4 & 5 & 6 & 7 & 8 & 9 & 10 \\
\midrule
$10^0$ & 89 & 75 & 37 & 38 & 23 & 20 & 18 & 18 & 14 & 14 \\
$10^3$ & $\infty$ & $\infty$ & 418 & 51 & 45 & 44 & 43 & 41 & 38 & 34 \\
$10^6$ & $\infty$ & $\infty$ & $\infty$ & $\infty$ & $\infty$ & $\infty$ & 372 & 42 & 34 & 34 \\
\bottomrule
\end{tabularx}

	\end{subtable}
	\caption{Iteration numbers for different choices of contrast for the MS-GFEM (top) and the MS-GFEM-R (bottom) 
	for the channelized coefficient in \cref{fig:coefficients} (right) with preconditioned Richardson iteration.}
	\label{tab:channel_iteration_numbers_richardson}
\end{table}

The second numerical experiment considers the domain $\Omega=(0,1)^2$, again with source term $f=1$
and homogeneous Dirichlet boundary conditions. A Cartesian mesh with 256 elements in each space dimension is used for the fine-scale discretization. 
The number of layers to form the oversampling domains is set to $\ell = 2$. We use $M = 16$ subdomains ($4$ in each space direction) and 
the coefficient is chosen as shown in \cref{fig:coefficients} (right). 
This coefficient is considered for three different values of the contrast, 
namely $1$ in the blue region and $10^j$ for $j \in \{0,3,6\}$ in the red region. Here $j = 0$ corresponds to the case of a globally constant coefficient. 

First, we consider the reciprocals of the eigenvalues of the local eigenvalue problems in \cref{eq:evpharmonicfullincludingconstant} and \cref{eq:evpharmonicringincludingconstant} of MS-GFEM and MS-GFEM-R, respectively, 
for three representative subdomains, shown in \cref{fig:channel_eigenvalues} (top row). 
The bottom row of \cref{fig:channel_eigenvalues} shows the corresponding coefficients on the oversampling domains $\omega^*$ and the oversampling rings
$\ringos$ are indicated by the black lines. The first subdomain (left) is located at the bottom left corner of the domain.
For high contrast, there are two and four large eigenvalues of MS-GFEM and MS-GFEM-R, respectively. 
For MS-GFEM this corresponds to the two connected high conductivity regions intersecting the overlapping region. 
MS-GFEM-R does not capture that the three vertical channels are connected, which is why there are four large eigenvalues. 
Similar observations can be made for the second and third subdomain. 
Observe that for these two subdomains the spectra of MS-GFEM-R are the same, since the coefficients on $R^*$ 
are the same. However, the spectra of MS-GFEM differ, because there are 
four and one connected high conductivity regions in the second and third subdomain, respectively.

Next, we consider the case of a preconditioned Richardson-type iteration in \cref{tab:channel_iteration_numbers_richardson}. For the MS-GFEM-R, using fewer than 8 basis functions per subdomain, and for the MS-GFEM, fewer than 4, we observe high iteration numbers or no convergence at all in the highest contrast setup.
This is to be expected for this particular choice of coefficient and is in agreement with the observations about the eigenvalues in \cref{fig:channel_eigenvalues}. 
Therefore, for this particular example, 8 basis functions are needed for the proposed method to capture the channels of the coefficient. 
Similar observations have been made for other ring- or slab-based spectral coarse spaces, see, e.g., the numerical experiments in \cite{Heinlein2019}. 

\subsection{Computational savings}
In the third numerical experiment, we illustrate the computational savings achieved by localizing the eigencomputations to rings. 
We consider a three-dimensional setup and consider only the coarse space construction on one representative subdomain $\omega$. To build the subdomain $\omega$, we start with a cube 
with $m$ elements in each space direction and extend this cube by one layer of fine elements.
 The corresponding ring domain is denoted by $R$. To construct the oversampling domains $\os$ and $\ringos$ we add $\ell$ oversampling layers to~$\om$ and~$\ring$, respectively. 
 {\color{black}
In \cref{fig:computational_savings} (left), the runtimes  (averaged over ten repetitions) for computing the first five eigenfunctions of~\cref{eq:evpharmonicfullincludingconstant} on $\omega^*$ are compared with those of~\cref{eq:evpharmonicringincludingconstant} on $\ringos$, as $m$ is increased. 
} To solve the eigenproblems numerically, we use the sparse eigensolver {\cAdd eigs in scipy, which is a wrapper to ARPACK}.
One observes that the runtime for the eigenproblem on the ring is significantly lower. For example, considering a subdomain $\om$ for  $m=25$ and minimal oversampling $\ell=1$, we have a speedup factor of about $8$ for the eigenproblem on the ring compared to the eigenproblem on the whole subdomain. {\color{black} For $m = 25$ and $\ell = 3$, the largest choice of the oversampling parameter, we still have a speedup factor of two. Note that, also for larger values of $\ell$, one expects that the gap between the runtime on the ring and the whole subdomain grows as $m$ is increased.}

There are two reasons for these computational advantages when considering rings: First, the eigenproblems on rings have fewer degrees of freedom. Second and more importantly, the saddle-point matrices that must be factorized when solving eigenproblems on rings have a better sparsity pattern than the corresponding matrices on the whole subdomains. More specifically, when computing on thin rings, the fill-in experienced by the sparse direct solver SuperLU, cf.~\cite{demmel1999superlu}, which is used within {\cAdd eigs}, is more comparable to that of a two-dimensional problem than to that of the present three-dimensional problem. 
 The improved fill-in can be observed in \cref{fig:computational_savings} (right), where we compare, for both methods, the average number of nonzero entries per row of the LU-factors of the saddle-point matrix in the eigenproblem. One observes that for rings that are thin compared to the size of the oversampling domain, the number of non-zero entries per row is significantly lower for the ring-based method.

\begin{figure}
	\includegraphics[width=.45\linewidth]{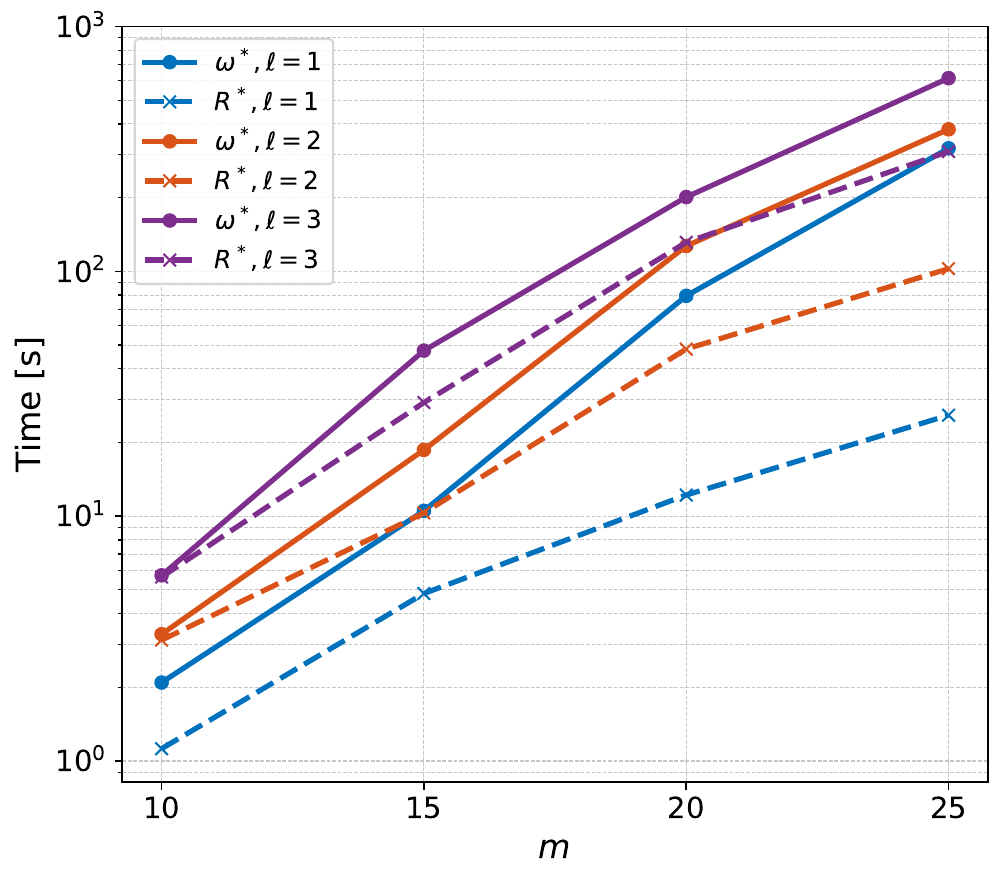}
	\hspace{.5cm}
	\includegraphics[width=.45\linewidth]{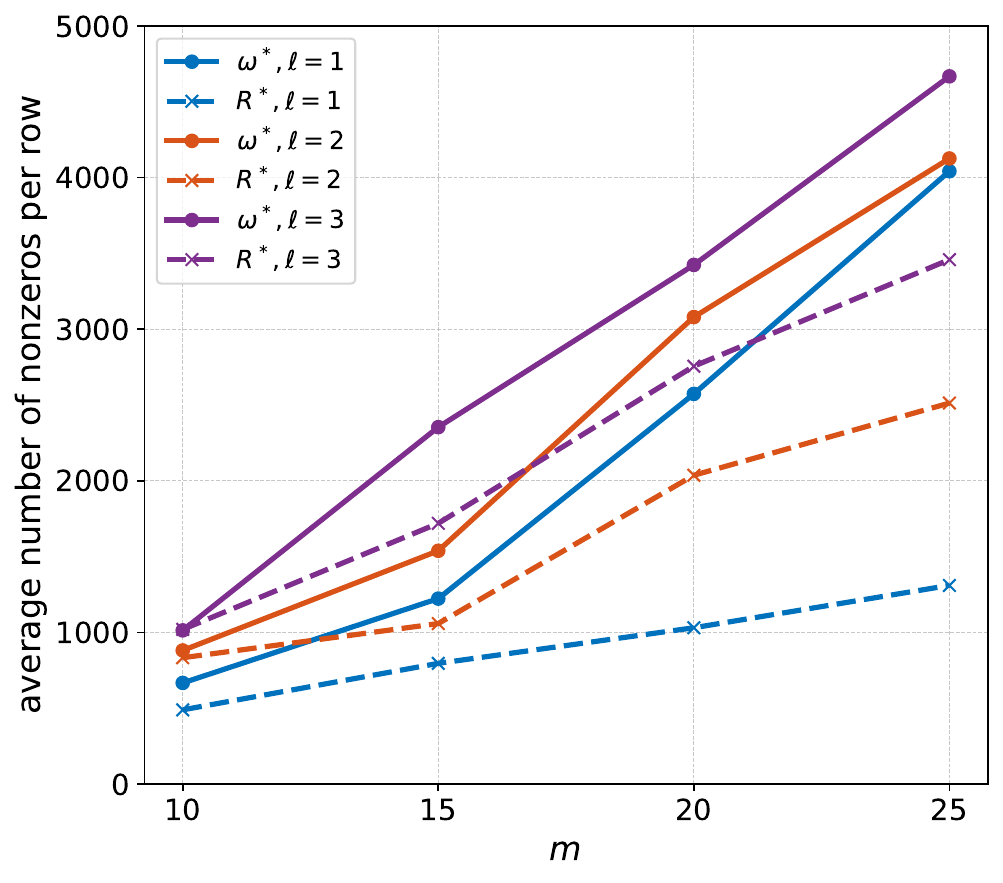}
	\caption{Computational times for eigenproblems (left) and average number of nonzero entries of LU-factors per row (right) for the MS-GFEM and the MS-GFEM-R as a function of the number~$m$ of fine-scale elements of the subdomains in each space direction.  The curves of the MS-GFEM and the MS-GFEM-R are marked with $\os$ and $\ringos$, respectively.}
	\label{fig:computational_savings}
\end{figure}

\section{Conclusion}

In this paper we have proposed a variant of the MS-GFEM where the spectral computations are performed only on rings around the boundary of the local subdomains.  This is in contrast to the classical MS-GFEM where these computations are performed on the whole subdomains. We have proved the (nearly-)exponential convergence of the proposed method in the number of local ansatz functions. Furthermore, the method can be used as a preconditioner where the approximation error of the multiscale method appears as the convergence rate. Therefore, the convergence result for the multiscale approximation implies a convergence result for the preconditioned iteration.
In practice, we have observed that localizing the spectral computations to thin rings leads to significant computational gains. This is because, compared to the classical MS-GFEM, the problem size of the local eigenproblems is reduced and one can benefit from a more favorable sparsity pattern of the underlying matrices. More specifically, for thin rings, the sparsity pattern of the matrices in $d$ dimensions is comparable to that of a problem of dimension~$d-1$, resulting in less fill-in for sparse direct solvers and ultimately faster spectral computations, as confirmed by our numerical experiments. Moreover, when used as a preconditioner, the number of iterations of the proposed variant (except for some tailor-made examples) is comparable to that of the classical MS-GFEM for the same coarse space dimension. This underscores the potential of the proposed MS-GFEM variant, and in the future we aim to provide a high-performance implementation of the method and examine it for problems beyond the elliptic model problem. 

\section*{Acknowledgments}
M.~Hauck acknowledges funding from the Deutsche Forschungsgemeinschaft\linebreak  (DFG, German Research Foundation) -- Project-ID 258734477 -- SFB 1173. 
The authors would like to thank Chupeng Ma (Great Bay University) for his helpful insights on the decay proof of the Kolmogorov $n$-width and Andreas Rupp (Saarland University) for fruitful discussions on the trace inequality on rings. The authors would also like to thank Alexander Heinlein (TU Delft) for his valuable remarks on spectral coarse spaces and their coefficient robustness.
\appendix

\bibliographystyle{alpha}
\bibliography{bib}
\end{document}